\documentclass{amsart}

\usepackage[utf8]{inputenc}
\usepackage{amssymb}
\usepackage[leqno]{amsmath}
\usepackage{latexsym}
\usepackage{amsthm}
\usepackage{enumitem}
\usepackage{graphicx}
\usepackage{epstopdf}
\usepackage{bigints}
\usepackage[all,cmtip]{xy}
\usepackage{esint}
\usepackage{stmaryrd}
\setlist[itemize]{leftmargin=4ex}

\theoremstyle{plain}
\newtheorem{theorem}{Theorem}

\newtheorem{lemma}[theorem]{Lemma}
\newtheorem{proposition}[theorem]{Proposition}

\theoremstyle{remark}
\newtheorem{remark}[theorem]{Remark}

\numberwithin{equation}{section}
\numberwithin{theorem}{section}
\newcommand{\N}{\mathbb{N}}
\newcommand{\R}{\mathbb{R}}
\newcommand{\Cleq}{\ensuremath{\lesssim}}  
\DeclareMathOperator{\diam}{diam}                              
\newcommand{\id}{\mathrm{Id}}                                  
\newcommand{\Mesh}{\mathcal{M}}
\newcommand{\Faces}[1]{\mathcal{F}_{#1}}

\newcommand{\FacesM}{\Faces{}}%{\Faces{\Mesh}}
\newcommand{\FacesMint}{\FacesM^i}
\newcommand{\FacesMbnd}{\FacesM^b}
\newcommand{\LagNod}[1]{\mathcal{L}_{#1}}
\newcommand{\LagNodM}[1]{\LagNod{#1}}
\newcommand{\LagNodMint}[1]{\LagNodM{#1}^i}
\newcommand{\Avg}[2]{\left \{\!\!\left \{ #1 \right \}  \!\!\right \}}%_{#2}}
\newcommand{\Jump}[2]{\left \llbracket #1 \right \rrbracket}%_{#2}}
\newcommand{\AvgOper}[1]{A_{#1}}
\newcommand{\Shape}{\gamma}
\newcommand{\nF}{n}

\newcommand{\degree}{p}
\newcommand{\Poly}[1]{\mathbb{P}_{#1}}                            
\newcommand{\Polyell}[1]{S_{#1}^0}%{\Poly{#1}(\Mesh)}             
\newcommand{\Lagr}[1]{S_{#1}^1}              
\newcommand{\CRS}[1]{\mathit{CR}_{#1}}%(\Mesh)}  
\newcommand{\hct}{\mathit{HCT}}
\newcommand{\HCT}{\hct}%(\Mesh)}
\DeclareMathOperator{\Div}{div}   
\DeclareMathOperator{\DivM}{\Div_{\Mesh}}
\newcommand{\Grad}{\nabla}
\DeclareMathOperator{\GradM}{\nabla_{\Mesh}}             
\DeclareMathOperator{\SymGrad}{\varepsilon}
\DeclareMathOperator{\SymGradM}{\SymGrad_{\Mesh}}

\DeclareMathOperator{\HessM}{D^2_{\Mesh}}                
\DeclareMathOperator{\Lapl}{\Delta}

\newcommand{\adj}[1]{{#1}^\star}
\newcommand{\BubbOper}[1][1]{B_{#1}}

\newcommand{\snorm}[1]{\left \lvert #1 \right \rvert}          
\newcommand{\norm}[1]{\| #1 \|}                     
\newcommand{\enorm}[1]{\snorm{#1}} 

\newcommand{\opnorm}[3]{\norm{#1}_{\ifx#2#3\mathcal{L}(#2)\else\mathcal{L}(#2,#3)\fi}}       
\newcommand{\Domain}{\Omega}
\newcommand{\Lebp}[2]{L^{#1}(#2)}
\newcommand{\Leb}[1]{\Lebp{2}{#1}}                                
\newcommand{\SobH}[1]{H^1_0(#1)}
\newcommand{\SobHD}[1]{H^{-1}(#1)}
\newcommand{\SobbH}[1]{H^2_0(#1)}
\newcommand{\App}{P}   
\newcommand{\app}{M}										   
\newcommand{\smt}{E}
\newcommand{\Ritz}{\Pi}                                          
\newcommand{\Vext}{\widetilde{V}}                         
\newcommand{\aext}{\widetilde{a}}
\newcommand{\bext}{\widetilde{b}}
\newcommand{\Cqopt}{C_{\mathrm{qopt}}}
\newcommand{\Cstab}{C_{\mathrm{stab}}} 
\newcommand{\bCons}{d}
\newcommand{\CR}{\mathit{CR}}
\newcommand{\BS}[1]{#1_\mathrm{BS}}
\newcommand{\SIP}{\mathrm{sip}}
\newcommand{\NIP}{\mathrm{nip}}
\newcommand{\VAR}{\mathrm{var}}
\newcommand{\HL}{\mathrm{HL}}
\begin{document}

%---TITOLO-------------------------%

\title[Quasi-optimal nonconforming methods III]
{Quasi-optimal nonconforming methods
 for symmetric elliptic problems.
 III -- DG and other interior penalty methods}

%---AUTORI-------------------------%

\author[A.~Veeser]{Andreas Veeser}
\author[P.~Zanotti]{Pietro Zanotti}

%---KEYWORDS E CLASSIFICAZIONE-----%

\keywords{}
%\subjclass[1991]{  }
%\date{  }

%---ABSTRACT-----------------------%

\begin{abstract}
We devise new variants of the following nonconforming finite element methods: DG methods of fixed arbitrary order for the Poisson problem, the Crouzeix-Raviart interior penalty method for linear elasticity, and the quadratic $C^0$ interior penalty method for the biharmonic problem. Each variant differs from the original method only in the discretization of the right-hand side. Before applying the load functional, a linear operator transforms nonconforming discrete test functions into conforming functions such that stability and consistency are improved. The new variants are thus quasi-optimal with respect to an extension of the energy norm. Furthermore, their quasi-optimality constants are uniformly bounded for shape regular meshes and tend to $1$ as the penalty parameter increases. 
\end{abstract}

\maketitle

%---SEZIONE 1----------------------%
%
%
\section{Introduction}
\label{S:introduction}
%
% short intro
This article is the third in a series on quasi-optimal nonconforming methods for linear and symmetric elliptic problems. Here we apply the framework developed in the first part \cite{Veeser.Zanotti:17p1} to design and analyze quasi-optimal interior penalty methods.
% with the help of tools devised in \cite{Veeser.Zanotti:17.II}.
Let us illustrate our construction and main results in the case of approximating the Poisson problem with discontinuous linear elements via the symmetric interior penalty (SIP) method, which was first studied by Wheeler \cite{Wheeler:78} and Arnold \cite{Arnold:82}.

% first order sip for Poisson
Let $u \in \SobH{\Domain}$ be the weak solution of the Poisson problem
\begin{equation}
\label{intro-Poisson}
 -\Delta u = f \text{ in }\Omega,
\qquad
 u = 0 \text{ on }\partial\Omega
\end{equation}
and let $\Mesh$ be a simplicial, face-to-face mesh of the domain $\Omega\subseteq\R^d$, $d\in\N$. We write $\Sigma$ for its skeleton, $\Polyell{1}$ for the space of discontinuous $\Mesh$-piecewise affine functions and use standard notation for piecewise gradients, jumps, averages, local meshsizes etc.\ (cf.\ \S\ref{S:simplices-meshes} below). The SIP approximation $U \in \Polyell{1}$ solves the discrete problem
\begin{equation}
\label{intro-sip-problem}
\forall \sigma \in \Polyell{1}
\qquad
b(U, \sigma)
=
\int_\Domain f \sigma
\end{equation}
where $f \in \Leb{\Domain}$, the bilinear form $b:= b_1 + b_2$
%on $\Polyell{1} \times \Polyell{1}$
is given by
%\begin{subequations}
\begin{align*}
%
%\label{intro-sip-b1}
	\nonumber
	& b_1(s, \sigma)
	:=
	\int_\Omega \GradM s \cdot \GradM \sigma 
	- 
	\int_\Sigma \Avg{\Grad s }{F}\cdot \nF \Jump{\sigma}{F},
\\
	%
	%\label{intro-sip-b2}
	\nonumber
	& b_2(s, \sigma)
	:=
	\int_\Sigma \dfrac{\eta}{h}  \Jump{s}{F} \Jump{\sigma}{F}
	- \int_\Sigma %\left(
	 \Jump{s}{F}\Avg{\Grad \sigma }{F}\cdot \nF,
	  %\right),  
\end{align*}
%\end{subequations}
and the penalty parameter $\eta>0$ is so large that $b$ is coercive. Replacing $s$ by $u\in\SobH{\Domain}$, we see that
\begin{equation}
\label{intro-sip-consistency}
 u \in H^2(\Domain) \implies \forall \sigma \in \Polyell{1} \;\,
 b_1(u,\sigma) = \int_\Domain f \sigma,
\quad\text{while}\quad
 \forall \sigma \in \Polyell{1} \;\,
 b_2(u, \sigma) = 0.
\end{equation}
Hence, $b_2$ establishes symmetry and coercivity, without impairing the \textit{consistency} provided by $b_1$. For shape regular meshes, one therefore can derive the following abstract error bound; cf.\ 
%\cite[Lemma 10.5.18]{Brenner.Scott:08},
Di Pietro and Ern \cite[Theorem 4.17]{DiPietro.Ern:12} and Gudi~\cite[\S3.2]{Gudi:10}:
\begin{equation}
\label{intro-sip-estimate}
 \enorm{u-U}_{1;\eta}
 \Cleq
 \inf \limits_{s \in \Polyell{1}} \Big(
  \enorm{u-s}_{1;\eta}^2 + \mathrm{AG}(u-s)^2
  %\sum_{F \in \FacesM} h_F \norm{\Avg{\Grad(u-s)}{F}}^2_{\Leb{F}}
 \Big)^{\frac{1}{2}},
\end{equation}
where the norm
\begin{equation*}
 \enorm{v}_{1;\eta}^2
 :=
 \int_\Omega |\GradM v|^2
  +
 \int_\Sigma \frac{\eta}{h} |\Jump{v}{F}|^2,
\quad
 v \in \SobH{\Domain}+\Polyell{1},
\end{equation*}
extends the energy norm associated with \eqref{intro-Poisson} and is augmented with
\begin{equation*}
 \mathrm{AG}(v)^2
 := \quad
 \int_\Sigma \frac{h}{\eta}\snorm{\Avg{\Grad v}{F}}^2
 \quad\text{or}\quad
 \sum_{K\in\Mesh} h_K^2 \inf_{c\in\R} \norm{\Delta v - c}_{\Leb{K}}^2.
\end{equation*}
on the right-hand side.
% deficiency of sip (bound) 
While \eqref{intro-sip-estimate} implies convergence of optimal order, the augmentation is an important difference to C\'{e}a's lemma. Indeed, since it is not meaningful for a generic solution in $\SobH{\Domain}$, it cannot be bounded by the best error $\inf_{s \in \Polyell{1}} \norm{u-s}$ and, in addition, it restricts the applicability of \eqref{intro-sip-estimate}. Notice that also the extension of $b_1$ underlying \eqref{intro-sip-consistency} and the right-hand side in the discrete problem \eqref{intro-sip-problem} have similar issues with generic instances of \eqref{intro-Poisson}. These observations suggest that the \emph{stability} of the SIP method \eqref{intro-sip-problem} is impaired.  More precisely, if, e.g., the right-hand side cannot be boundedly extended to $H^{-1}(\Domain)=\SobH{\Domain}$, then $\enorm{U}_{1,;\eta}$, or the error $\enorm{u-U}_{1;\eta}$, cannot be bounded in terms of $\norm{f}_{\SobHD{\Domain}}$. Since this `full stability' is necessary for removing the augmentation $\mathrm{AG}$ from \eqref{intro-sip-estimate}, we thus expect that the SIP method \eqref{intro-sip-problem} is not $\enorm{\cdot}_{1;\eta}$-\emph{quasi-optimal} and so does not always fully exploit the approximation potential offered by its discrete space $\Polyell{1}$. This suspect is confirmed by Remark~4.9 in the first part \cite{Veeser.Zanotti:17p1} of this series.

% the modified sip
In order to achieve quasi-optimality, we consider the following variant of the discrete problem \eqref{intro-sip-problem}:
find $U_\smt \in \Polyell{1}$ such that
\begin{equation}
\label{intro-sip-modproblem}
 \forall \sigma \in \Polyell{1}
\quad
 b(U_\smt, \sigma)
 =
 \langle f,  \smt \sigma \rangle,
\end{equation}
where the linear operator $\smt: \Polyell{1} \to \SobH{\Domain}$ to be specified enables $f \in H^{-1}(\Omega)$.  If we require that the means on internal faces are conserved as in Badia et al.\ \cite{Badia.Codina.Gudi.Guzman:14},
\begin{equation}
\label{intro-E-faces}
\forall \sigma \in \Polyell{1}, \, F \in \FacesMint
\quad
\int_F \smt \sigma 
=
\int_F \Avg{\sigma}{F},
\end{equation}
then piecewise integrating by parts twice shows
\begin{equation*}
 \forall s,\sigma \in \Polyell{1}
\quad
 b_1(s,\sigma)
 =
 \int_\Sigma  \Jump{\Grad s }{F}\cdot \nF \Avg{\sigma}{F}
 =
 \int_\Omega \nabla_\Mesh s \cdot \nabla (\smt\sigma).
\end{equation*}
Interestingly, the right-hand side provides a new extension $\bext_1$ of $b_1$ onto $\SobH{\Domain}$ which improves upon the first part of \eqref{intro-sip-consistency} in that
\begin{equation*}
 \forall u \in \SobH{\Domain}, \, \sigma \in \Polyell{1}
\quad
 \bext_1(u,\sigma) = \langle f, \smt\sigma \rangle.
\end{equation*}
In order to construct an `$\SobH{\Domain}$-smoothing operator' that satisfies \eqref{intro-E-faces} and is computionally feasible, we extend a similar operator devised in the second part \cite{Veeser.Zanotti:17p2} of this series, ensuring that its operator norm $\opnorm{\smt}{\Polyell{1}}{\SobH{\Domain}}$ is bounded in terms of the shape coefficient $\Shape_\Mesh$ of $\Mesh$.

Exploiting the improved stability and consistency properties of \eqref{intro-sip-modproblem}, the abstract theory of \cite{Veeser.Zanotti:17p1} then yields
\begin{equation*}
\label{intro-sip-modestimate}
 \enorm{u-U_\smt}_{1;\eta}
 \leq
 \left(
  1 + C\eta^{-1}
 \right)^{\frac{1}{2}}
 \inf \limits_{s \in \Polyell{1}} \enorm{u-s}_{1;\eta},
\end{equation*}
where $C$ depends on $d$ and $\Shape_\Mesh$ and $\eta$ is sufficiently large. Notably, as $\eta\to\infty$, the discontinuous space $\Polyell{1}$ is replaced by the space $\Lagr{1}$ of continuous piecewise affine functions and we end up exactly in C\'ea's Lemma for the conforming Galerkin method with $\Lagr{1}$.

% comparison with CR
It is worth comparing with the quasi-optimal Crouzeix-Raviart method for \eqref{intro-Poisson} of the second part \cite{Veeser.Zanotti:17p2} of this series. Thanks to the coupling between Crouzeix-Raviart elements, $b_1$ becomes symmetric and there is no need for $b_2$ and penalization. As a consequence, the ensuing quasi-optimality constant equals the operator norm  with respect to the piecewise energy norm of the smoothing operator $\smt$.

% organization of the paper
The rest of this article is organized as follows. Section \ref{S:theory} provides a brief summary of the abstract results in \cite{Veeser.Zanotti:17p1} to be used here.  In Section \ref{S:Applications}, we introduce new variants of various interior penalty methods and prove their quasi-optimality. Firstly, we design quasi-optimal DG methods of arbitrary fixed order for the Poisson problem, covering also the setting illustrated in this introduction. Secondly, we devise a quasi-optimal Crouzeix-Raviart interior penalty method for linear elasticity and establish a robust error bound for it in the nearly-incompressible regime. Lastly, we conclude with a quasi-optimal variant of the quadratic $C^0$-interior penalty method for the biharmonic problem.

% extensions
In these examples, we consider polyhedral domains with Lipschitz boundaries and homogeneous essential boundary conditions. An application of the presented approach to more general domains and boundary conditions as well as numerical investigations will be presented elsewhere.

\section{Stability and consistency for quasi-optimality}
\label{S:theory}
%
%
% intro
We briefly summarize the characterization of quasi-optimality in \cite{Veeser.Zanotti:17p1}, adopting the approach to nonconforming consistency corresponding to the so-called
second Strang lemma.

% continuous problem
A linear and symmetric elliptic problem can be written in the following abstract form: given $\ell \in V'$, find $u\in V$ such that
\begin{equation}
\label{ex-prob}
\forall v \in V
\quad
a(u, v) 
= 
\langle \ell, v \rangle,
\end{equation} 
where $V$ is an infinite-dimensional Hilbert space with scalar product $a(\cdot, \cdot)$, $V'$ is its (topological) dual space, and $\left\langle \cdot, \cdot\right\rangle$ stands for the dual pairing of $V$ and $V'$. We write
$\norm{\cdot} = \sqrt{a(\cdot,\cdot)}$ for the
\emph{energy norm}, which induces the \emph{dual energy norm}
$\norm{\ell}_{V'} := \sup_{v\in V, \norm{v} = 1} \langle\ell,v\rangle$ on $V'$. 
% Riesz isometry A
Problem \eqref{ex-prob} is well-posed in the sense of Hadamard and, introducing the Riesz isometry $A:V \to V'$, $v \mapsto a(v, \cdot)$, we have $u=A^{-1}\ell$ with $\norm{u} = \norm{\ell}_{V'}$.

% discrete problem
We shall design quasi-optimal methods $\app: V' \to S$ with discrete problems of the following form: given $\ell \in V'$, find $\app\ell \in S$ such that
\begin{equation}
\label{disc-prob}
\forall \sigma \in S
\quad
b(\app \ell, \sigma)
=
\langle \ell, \smt \sigma \rangle,
\end{equation}
where $S$ is a finite-dimensional linear space, $b : S \times S \to \R$ a nondegenerate bilinear form, $\smt$ a linear operator from $S$ to $V$, and $\left\langle \cdot, \cdot\right\rangle$ stands also for the pairing of $S$ and $S'$. Although we do not require $S \subset V$, the operator $E$ ensures that the method $M$ is \emph{entire}, i.e.\ defined for all $\ell \in V'$. In light of \cite[Remark~2.4]{Veeser.Zanotti:17p1}, this is a necessary condition for the kind of quasi-optimality we are interest in. We refer to $\smt$ as a \textit{smoothing operator} or \textit{smoother}, because $S\not\subset V$ often arises for the lack of smoothness. Moreover, we identify the operator $\app$ with the triplet $(S,b,\smt)$, ignoring some slight ambiguity; cf.\  \cite[Remark~2.2]{Veeser.Zanotti:17p1}.

% continuous and discrete problem
The relationship between continuous and discrete problem is illustrated by the commutative diagram in Figure~\ref{F:EntireNonconformingMethods}.
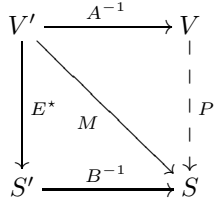
\begin{figure}	
	\[
	\xymatrixcolsep{4pc}		% Size of column separation
	\xymatrixrowsep{4pc}		% Size of row separation
	\xymatrix{
		V'   									% Space 1
		\ar[d]^{\smt^\star}
		\ar[r]^{A^{-1}}       % Right arrow
		\ar[rd]_{\app} 					% Right-down arrow 
		& V 									% Space 2
		\ar@{-->}[d]^{\App} \\		% Down arrow (with label)
		S'										% Space 3
		\ar[r]^{B^{-1}}				% Right arrow
		& S}               		% Space 4
	\]
	\caption{\label{F:EntireNonconformingMethods}	Diagram with operators $A$, $B$, $\smt$, nonconforming method $\app = (S, b, \smt)$, and induced approximation operator $\App$.}
\end{figure}
This diagram introduces
\begin{itemize}
	\item the adjoint $\adj{\smt} :V' \to S'$ given by $\left\langle \adj{\smt} \ell, \sigma \right\rangle =  \left\langle \ell , \smt\sigma \right\rangle $ for $\ell \in V'$, $\sigma \in S$,
	\item the invertible map $B:S \to S'$, $s \mapsto b(s,\cdot)$,
	\item  the approximation operator $\App := \app A$
\end{itemize} 
and illustrates the representations
\begin{equation}
\label{M=}
\app = B^{-1} \adj{\smt}
\qquad \mathrm{and} \qquad
\App = B^{-1}\adj{\smt} A.
\end{equation}

% Extended energy norm and quasi-optimality
A solution $u$ of \eqref{ex-prob} is thus approximated by $\app\ell$ with $\ell = Au$, that is, by $\App u$. To assess the quality of this approximation, we assume that $a$ can be extended to a scalar product $\aext$ on the sum $\Vext:=V+S$ and quantify the error with the \emph{extended energy norm}
\begin{equation*}
\norm{\cdot} := \sqrt{\aext(\cdot, \cdot)}
\quad\text{on }\Vext,
\end{equation*}
using the same notation as for the original one.  The best approximation error within $S$ to $u$ is then $\inf_{s\in S} \norm{u-s}$ and admitted by the $\aext$-orthogonal projection $\Ritz_S$ onto $S$.  We say that the method $\app$ is \emph{quasi-optimal} (for Problem \eqref{ex-prob} with respect to the extended energy norm) if there exists a constant $C \geq 1$ such that
\begin{equation}
\label{Cqopt-def}
\forall u \in V
\qquad
\norm{ u - \App u }
\leq
C \inf_{s\in S} \norm{u-s}.
\end{equation}	
The associated quasi-optimality constant $\Cqopt$ of $\app$ is then the smallest constant with this property. Notice that \eqref{Cqopt-def} involves all exact solutions of \eqref{ex-prob}, not only certain smooth ones.%The following result from \cite{Veeser.Zanotti:17} is our departure point in the design of quasi-optimal methods.
\begin{theorem}[Stability, consistency, and quasi-optimality]
\label{T:abs-qopt}
Given a nonconforming method $\app = (S,b,\smt)$ for \eqref{ex-prob} and an extended scalar product $\aext$, introduce the bilinear form $\bCons: V \to \R$ by
\begin{equation*}
 \bCons(v,\sigma) := b(\Ritz_S v,\sigma) - a(v,\smt\sigma).
\end{equation*}
Then:
\begin{itemize}
\item[(i)] $\app$ is bounded, or fully stable, with 
\begin{equation*}
		%\label{Cstab-with-smoothing}
 \Cstab
 :=
 \opnorm{\app}{V'}{S}
 =
 \sup_{\sigma\in S} %\, 
  \frac{\norm{\smt\sigma}}{\sup_{s \in S, \norm{s} = 1} b(s,\sigma)}.
\end{equation*}
\item[(ii)] $\app$ is quasi-optimal if and only if it is fully algebraically consistent in that
\begin{equation*}
%\label{fa-consistency-with-E}
 \forall u \in S\cap V, \sigma \in S
\quad
 0
 =
 \bCons(u, \sigma)
 =
 b(u,\sigma) - a(u,\smt\sigma).  
\end{equation*}
\item[(iii)] If $\app$ is quasi-optimal, then its quasi-optimality constant 
		%		is 
		%		\begin{equation*}
		%		%\label{Cqopt=;smoothing}
		%		\Cqopt
		%		=
		%		\sup_{\sigma \in S} \,
		%		\frac{ \sup_{v\in V, s \in S, \norm{v+s} = 1} a(v,\smt\sigma) + b(s,\sigma) }
		%		{\sup_{s \in S, \norm{s} = 1} b(s,\sigma)}
		%		\end{equation*}
		%		and 		
satisfies
\begin{equation*}
 \Cstab
 \leq
 \Cqopt
 =
 \sqrt{1+\delta^2},
\end{equation*}
where $\delta \in [0,\infty)$ is the consistency measure given by the smallest constant in
\begin{equation*}
%\label{delta_S}
 \forall v \in V, \sigma \in S
\quad
 \snorm{\bCons(v,\sigma)}  
 \leq
 \delta \sup_{\hat{s}\in S, \norm{\hat{s}}=1} b(\hat{s},\sigma) \;
	\inf \limits_{s \in S} \norm{v - s}.
\end{equation*}
\end{itemize} 
\end{theorem}

\begin{proof}
Item (i) follows from \cite[Theorem~4.7]{Veeser.Zanotti:17p1}, while (ii) is a consequence of \cite[Theorem~3.2]{Veeser.Zanotti:17p1} and (i). Finally, the first part of \cite[Theorem~4.19]{Veeser.Zanotti:17p1} implies (iii).
\end{proof}

% Comments on Theorem CSQO
Some comments on Theorem \ref{T:abs-qopt} and consequences of its proof are in order. In (i) and (ii), the adverb `fully' stands for the fact that all (and not only certain) smooth instances of Problem \eqref{ex-prob} are involved.
% full stability
The built-in \emph{full stability} of the considered methods is a necessary condition for quasi-optimality. It has to be established by applying a smoothing operator $\smt$ before evaluating the load functional.  We refer to the constant $\Cstab$ as the stability constant of $\app$.

% full algebraic consistency
Notice that \emph{full algebraic consistency} does not actually depend on the extension $\aext$ of the scalar product $a$. In particular, it can be rephrased in the following manner: whenever an exact solution happens to be discrete, it has to be also the discrete solution.
% nonconforming Galerkin methods
Natural candidates for full algebraic consistency are \emph{nonforming Galerkin methods} satisfying
\begin{equation}
\label{NonConformingGalerkinMethod}
b_{|S_C \times S_C} = a_{|S_C \times S_C}
\quad\text{and}\quad
E_{|S_C} = \id_{S_C},
\end{equation}
where $S_C := S \cap V$ is the conforming subspace of $S$. Notice that this generalization of conforming Galerkin methods does not determine $b$ and $E$ if $S$ is truly nonconforming. Furthermore, it may be weaker than full algebraic consistency, involving also nonconforming discrete test functions.

% nonconforming consistency
While full algebraic consistency involves only the conforming part $S_C$ of the discrete space, the constant $\delta$ captures consistency properties of $\app$ for nonconforming directions in $S \setminus V$. We call $\app$ \emph{(algebraically) overconsistent} whenever $\bCons(\cdot,\cdot)$ vanishes, that is whenever the discrete bilinear form $b$ is $\aext(\cdot,\smt\cdot)$. In this case, $\Cstab=\Cqopt$; see \cite[Theorem~2.5]{Veeser.Zanotti:17p2}. The following simple consequence of the inf-sup theory implies that this appealing property requires a certain interplay of $S$, $V$, and $\aext$; cf.\ also \cite[Lemma~2.4]{Veeser.Zanotti:17p2}, which strengthens the statement to a characterization.
\begin{lemma}[Obstruction for nondegenerate $\aext(\cdot,\smt\cdot)$]
\label{L:a(.,E.)}
Let $S$, $V$, and $\aext$ be given as in Theorem \ref{T:abs-qopt} and assume that the intersection of $S$ and the $\aext$-orthogonal complement of $V$ in $V+S$ is nontrivial. Then, for any smoother $\smt:S\to V$, the bilinear form $\aext(\cdot,\smt\cdot)$ is degenerate.% and, consequently, there is no overconsistent method of the form $(S,\aext(\cdot,\smt\cdot)_{|S\times S},\smt)$.
\end{lemma}

For a further discussion of the aforementioned notions, their role and properties, we refer to \cite{Veeser.Zanotti:17p1} and \cite[\S2.2]{Veeser.Zanotti:17p2}. Here we continue by underlining that Theorem~\ref{T:abs-qopt} was formulated  with the following viewpoint: for quasi-optimality, the discrete bilinear form $b$ has to be a perturbation of $\aext(\cdot,\smt\cdot)$, which is fully algebraically consistent and affects the quasi-optimality via $\delta$. This viewpoint will be our guiding principle for constructing quasi-optimal interior penalty methods. It is therefore of interest to bound $\Cstab$ and $\delta$, connecting it to a well-known and important, but not yet mentioned constant.
\begin{remark}[Stability, consistency and inf-sup constants]
\label{R:stab-cons-infsup}
Let $\app = (S, b, \smt)$ be a nonconforming method. As $S$ is finite-dimensional, the nondegeneracy of $b$ entails that the inf-sup constant is positive:
\begin{equation*}
\label{b-infsup1}
 \alpha
 :=
 \inf_{\sigma \in S, \norm{\sigma}=1}
  \sup_{s \in S, \norm{s}=1} 
	b(s, \sigma)
 >0.
\end{equation*}
Then the definitions of $\Cstab$ and $\delta$ readily yield
\begin{equation}
\label{Cstab-and-||E||,delta-and-||bC||}
 \Cstab
 \leq
 \frac{\;\opnorm{\smt}{S}{V}}{\alpha}
\quad\text{and}\quad
 \delta 
 \leq 
 \frac{\gamma}{\alpha}
\end{equation} 
where $\gamma \geq 0$ verifies $\snorm{\bCons(s,\sigma)} \leq \gamma \inf_{s \in S}\norm{v-s} \norm{\sigma}$ for all $v \in V$ and $\sigma\in S$.  Hence, up to the inverse of the inf-sup constant $\alpha$, the constants $\Cstab$ and $\delta$ depend,  respectively, only on the smoothing operator $\smt$ and the bilinear form $\bCons$. It is worth noting that these bounds may be pessimistic; see \cite[Remark~2.7]{Veeser.Zanotti:17p2}.
\end{remark}

\section{Applications to interior penalty methods}
\label{S:Applications}
The goal of this section is to devise interior penalty methods that are based upon nonconforming finite elements and are quasi-optimal. In view of Theorem~\ref{T:abs-qopt} and Lemma~\ref{L:a(.,E.)}, we may achieve this by the following steps: given a continuous problem \eqref{ex-prob} and a nonconforming finite element space $S$,
\begin{itemize}
\item extend the scalar product $a$ to the sum $V+S$,
\item find a computationally feasible smoothing operator $\smt:S\to V$, possibly with $\smt_{|V \cap S} = \id_{V \cap S}$,
\item if necessary, use the bilinear form $\bCons$ to arrange that $b=\aext(\cdot,\smt\cdot)+\bCons$ is nondegenerate and has other optional properties like symmetry.
\end{itemize}
Denoting by $\varphi_1, \dots, \varphi_n$ the nodal basis of $S$, we consider a smoothing operator $\smt$ to be \emph{computationally feasible} if each $\smt\varphi_i$ is in some conforming finite element space and the number of elements in its support is bounded independently of $n$.

We shall carry out the aforementioned steps for three different settings, involving vector and fourth order problems as well as various couplings between elements (completely discontinuous, Crouzeix-Raviart, continuous). In each case the nondegeneracy of $b$ will be obtained by means of interior penalties.

\subsection{Simplicial meshes and (broken) function spaces}
\label{S:simplices-meshes}
We indicate Lebesgue and Sobolev spaces as usual, see, e.g., \cite{Brenner.Scott:08}, and adopt the following notations, mainly taken from \cite{Veeser.Zanotti:17p2}. %/, for simplicial meshes, broken Sobolev and polynomial functions.

% simplices
Given $n \in \{0, \dots, d\}$, an $n$-\emph{simplex} $C \subseteq \R^d$ is the convex hull of $n+1$ points $z_1, \dots, z_{n+1} \in \R^d$ spanning an $n$-dimensional affine space. The uniquely determined points $z_1, \dots, z_{n+1}$ are the vertices of $C$ and form the set $\LagNod{1}(C)$. If $n\geq1$, we let $\Faces{C}$ denote the $(n-1)$-dimensional faces of $C$, which are the $(n-1)$-simplices arising by picking $n$ distinct vertices from $\LagNod{1}(C)$.  Given a vertex $z \in \LagNod{1}(C)$, its barycentric coordinate $\lambda_z^C$ is the unique first order polynomial on $C$ such that $\lambda_z^C(y) = \delta_{zy}$ for all $y \in \LagNod{1}(C)$. Then  $0\leq\lambda_z^C\leq 1$ and $\sum_{z\in\LagNod{1}(C)} \lambda_z^C =1 $ in $C$ and, if $\alpha = (\alpha_z)_{z\in\LagNod{1}(C)} \in \N_0^{n+1}$ is multi-index,
\begin{equation}
\label{intregation-bary-coord}
 \int_C \prod_{z\in\LagNod{1}(C)} (\lambda_z^C)^{\alpha_z} = \frac{n!\alpha!}{(n+|\alpha|)!} \snorm{C},
\end{equation}
where $\snorm{C}$ stands also for the $n$-dimensional Hausdorff measure in $\R^d$. We write $h_C := \diam(C)$ for the diameter of $C$, $\rho_C$ for the diameter of its largest inscribed $n$-dimensional ball, and $\Shape_C$ for its shape coefficient $\Shape_C :=  h_C / \rho_C$.

% meshes
Let $\Mesh$ be a simplicial, face-to-face \emph{mesh} of some open, bounded, connected and polyhedral set $\Omega\subset\R^d$ with Lipschitz boundary $\partial\Omega$.  More precisely, $\Mesh$ is a finite collection of $d$-simplices in $\R^d$ such that
$\overline{\Domain} = \bigcup_{K \in \Mesh} K$ and the intersection of two arbitrary elements $K_1, K_2 \in \Mesh$ is either empty or an $n$-simplex with $n\in\{0 \dots, d\}$ and $\LagNod{1}(K_1 \cap K_2) = \LagNod{1}(K_1)\cap \LagNod{1}(K_2)$. We let $\FacesM :=  \bigcup_{K \in \Mesh} \Faces{K}$ denote
the $(d-1)$-dimensional faces of $\Mesh$ and distinguish between boundary faces
$\FacesMbnd := \{ F \in \FacesM \mid F \subseteq \partial\Omega \}$ and interior faces $\FacesMint := \FacesM \setminus \FacesMbnd$.  Moreover, let $\Sigma := \cup_{F\in\FacesM} F$ be the skeleton of $\Mesh$ and, fixing a unit normal $n_F$ for each interior face $F\in\FacesMint$, extend the outer normal $n$ of $\partial\Omega$ to $\Sigma$ by $n_{|F} = n_F$ for $F\in\FacesMint$. The ambiguity of $n_F$ is insignificant to our discussion. The meshsize $h$ on $\Sigma$ is given by $h_{|F} = h_F$ for all $F\in\FacesM$ and the shape coefficient of $\Mesh$ is 
\begin{equation*}
\label{shape-Mesh}
\Shape_\Mesh:=
\max \limits_{K \in \Mesh} \Shape_K.
\end{equation*}

% broken Sobolev spaces, umps and averages
For $k\in\N$, the broken Sobolev space of order $k$ is 
\begin{equation*}
 H^k(\Mesh)
 :=
 \{ v\in\Leb{\Domain} \mid \forall K \in \Mesh \; v_{|K} \in H^k(K) \}.
\end{equation*}
If $v\in H^k(\Mesh)$, we use the subscript $\Mesh$ to indicate the piecewise variant of a differential operator. For instance, $\GradM v$ is given by $(\GradM v)_{|K} := \nabla(v_{|K})$ for all $K \in \Mesh$. Jumps and averages are defined as follows. Given an interior face $F \in \FacesMint$, let $K_1, K_2 \in \Mesh$ be the two elements such that $F=K_1 \cap K_2$ and the outer normal of $K_1$ coincides with $n$. Set
\begin{subequations}
\label{jumps+averages}
\begin{equation}
\label{ja:interior-faces}
 \Jump{v}{F}
 := 
 v_{|K_1} - v_{|K_2},
\quad
 \Avg{v}{F}
 :=
 \frac{1}{2} \left( v_{|K_1} + v_{|K_2} \right)
\quad
 \text{on }F.
\end{equation}
The fact that the signs of $\nF$ and $\Jump{v}{F}$ depend on the ordering of $K_1$ and $K_2$ will be insignificant to our discussion. It will be convenient to extend these definitions on $\partial\Omega$. Given  $F \in \FacesMbnd$, let $K \in \Mesh$ be the element such that $F = K \cap \partial \Domain$ and set
\begin{equation}
\label{jumps-avgs@bdy}
 \Jump{v}{F}
 :=
 \Avg{v}{F}
 :=
 v_{|K}
\quad
 \text{on }F.
\end{equation}
\end{subequations}
In this notation, piecewise integration by parts reads as follows: if $v,w \in H^1(\Mesh)$ and $j\in\{1,\dots,d\}$, then
\begin{equation}\begin{aligned}
\label{pwIbP}
 \int_{\Omega} (\partial_{j,\Mesh} v) w
 &-
 \int_{\Sigma\setminus\partial\Omega} \Jump{v}{F} \Avg{w}{F} \nF \cdot e_j
\\ &=
 -\int_{\Omega} v (\partial_{j,\Mesh} w)
 +
 \int_{\Sigma\setminus\partial\Omega}
  \Avg{v}{F}  \Jump{w}{F} \nF \cdot e_j
 +
 \int_{\partial\Omega} vw \, \nF \cdot e_j.
\end{aligned}\end{equation}
Notice that the surface integrals are independent of the orientation of $n$ and that, e.g., the singular part of the distributional derivative $\partial_jv$ is represented by means of the negative jumps $-\Jump{v}{F}$, $F\in\FacesMint$.

% polynomials on simplices
Given $\degree\in\N_0$, we write $\Poly{\degree}(C)$ for the linear space of \emph{polynomials} on the $n$-simplex $C$ with (total) degree $\leq\degree$. Consider $p\in\N$, excluding the trivial case $p=0$. A polynomial in $\Poly{\degree}(C)$ is determined by its point values at the Lagrange nodes 
$\LagNod{\degree}(C)$ of order $\degree$, which, for $\degree\geq2$, are given by $\big\lbrace x \in C \mid \forall z \in \LagNod{1}(C) \;
p\lambda_z^C (x) \in \N_0 \big\rbrace$. We let $\Psi_{C,z}^\degree$, $z\in\LagNod{\degree}(C)$, denote the associated nodal basis in $\Poly{\degree}(C)$ given by $\Psi_{C,z}^\degree(y) = \delta_{zy}$ for all $y,z\in\LagNod{\degree}(C)$. The Lagrange nodes are nested in that $\LagNod{\degree}(F) = \LagNod{\degree}(C) \cap F$ for any face $F\in\Faces{C}$. Thus, the restriction $P_{|F}$ of $P\in\Poly{\degree}(C)$ is determined by the `restriction' $\LagNod{\degree}(C) \cap F$ of the Lagrange nodes and we have
$\Psi_{C,z}^\degree{}_{|F} = \Psi_{F,z}^\degree$ for all $z\in\LagNod{\degree}(F)$.

% polynomial spaces
Given $k,\degree\in\N_0$, the space of functions that are \emph{piecewise polynomial} with degree $\leq\degree$ and are in $H^k_0(\Omega)$ is
\begin{equation}
\label{Skp}
 S^k_\degree
 :=
 \left\{ s \in H^k_0(\Omega) \mid
  \forall K\in\Mesh \; s_{|K} \in \Poly{\degree}(K)
 \right\}.
\end{equation}
The cases $p\in\N$ with $k\in\{0,1\}$ are of particular interest.

% discontinuous basis
Consider first $\Polyell{\degree}$ with $p\in\N$ and extend each $\Psi_{K,z}^\degree$ outside of $K\in\Mesh$ by $0$. The functions $\{\Psi_{K,z}^\degree\}_{K\in\Mesh,z\in\LagNod{\degree}(K)}$ form a basis of $\Polyell{\degree}$ with $\Psi_{K,z}{}_{|K'}(z') = \delta_{K,K'}\delta_{z,z'}$ for $K,K'\in\Mesh$ and $z\in\LagNod{\degree}(K)$, $z'\in\LagNod{\degree}(K')$, which amounts to distinguishing Lagrange nodes from different elements.

% continuous basis
The construction of a basis of $\Lagr{\degree}$ is a little more involved. Here, identifying coinciding Lagrange nodes, we set $\LagNod{\degree} := \cup_{K\in\Mesh} \LagNod{\degree}(K)$ as well as $\LagNodMint{\degree} := \LagNod{\degree} \setminus \partial\Omega$, and write $\Phi_z^\degree$, $z\in\LagNod{\degree}$, for the function given piecewise by $\Phi_z^\degree{}_{|K} := \Psi_{K,z}^\degree$ if $z\in K$ and $\Phi_z^\degree{}_{|K} := 0$ otherwise. Then the nestedness of Lagrange nodes implies: $\{\Phi_z^\degree\}_{z \in \LagNodMint{\degree}}$ is a basis of $\Lagr{\degree}$ satisfying $\Phi^p_z(y) = \delta_{zy}$ for all $y,z\in\LagNodMint{\degree}$. In connection with these basis functions, the following subdomains are useful. Let $\omega_z:= \bigcup_{K' \ni z} K'$ be the star around $z\in\LagNod{\degree}$ and let $\omega_K := \bigcup_{K'\cap K \neq \emptyset} K'$ be the patch around $K \in \Mesh$. Since $\partial\Omega$ is Lipschitz, stars are face-connected in the sense of \cite{Veeser:16}: given $z\in\LagNod{\degree}$ and any pair $K,K'\in\Mesh$ with $z\in K \cap K'$, there exists a path $\{K_i\}_{i=1}^n\subset\Mesh$ of elements containing $z$ such that $K_1=K$, $K_n=K'$, and each $K_i\cap K_{i+1} \in\FacesMint$.

% generic constants
If not specified differently, $C_*$ stands for a function which is not necessarily the same at each occurrence and depends on a subset $*$ of $\{d,\Shape_\Mesh,\degree\}$, increasing in $\Shape_\Mesh$ and $\degree$ if present. For instance, we have, for $K,K' \in \Mesh$,
\begin{equation}
\label{adjacent-elements}
 K \cap K' \neq \emptyset
\quad\implies\quad
 \snorm{K} \leq C_{\Shape_\Mesh} \snorm{K'}
\text{ and }
 h_K \leq C_{\Shape_\Mesh}\rho_{K'}
\end{equation}
and, for $p\in\N$, $K \in \Mesh$, and $z\in\LagNod{\degree}(K)$,
\begin{equation}
\label{Lagrange-basis:scaling}
 c_{d,\degree} |K|^{\frac{1}{2}} h_{K}^{-1}
\leq
\norm{\nabla \Psi_{K,z}^\degree}_{\Leb{K}}
\leq
C_{d,\degree} |K|^{\frac{1}{2}} \rho_{K}^{-1}.
\end{equation}
%thanks to the fact that Lagrange elements of order $p$ are affine equivalent. 
If there is no danger of confusion, $A \leq C_* B$ may be abbreviated as $A \Cleq B$.
\subsection{Quasi-optimal DG methods for the Poisson problem}
\label{S:DG-for-Poisson}
In this subsection we devise quasi-optimal DG methods for the Poisson problem, covering the results illustrated in the introduction \S\ref{S:introduction}.

Let $\Omega$ and $\Mesh$ be as in \S\ref{S:simplices-meshes} and, with $\eta\geq0$, define
\begin{equation}
\label{a_eta}
(v,w)_{1;\eta}
:=
\int_{\Omega} \GradM v \cdot \GradM w
+
\sum_{F\in\FacesM} \frac{\eta}{h_F} \int_F \Jump{v}{F} \Jump{w}{F},
\quad
\snorm{v}_{1;\eta} := (v,v)_{1;\eta}^{\frac{1}{2}}
\end{equation}
on $H^1(\Mesh)$ and abbreviate $(\cdot,\cdot)_{1;0}$ to $(\cdot,\cdot)_1$. Recalling \eqref{Skp}, we consider
\begin{equation}
\label{DG-for-Poisson:setting}
\begin{gathered}
 V = \SobH{\Domain},
\quad
 S = \Polyell{\degree} \text{ with }\degree \in \N,
\quad
 \aext = (\cdot,\cdot)_{1;\eta}
 \text{ on } \Vext = \SobH{\Domain} + \Polyell{\degree}. 
\end{gathered}
\end{equation}
Then $\aext$ is a scalar product for $\eta>0$ and the abstract problem \eqref{ex-prob} provides a weak formulation of \eqref{intro-Poisson}. Our setting has two parameters: the polynomial degree $\degree$ and the scaling factor $\eta$ of the jumps. The latter will be also the penalty parameter and is essentially free to be specified by the user. In order to keep notation simple, we shall sometimes suppress the dependencies on $\degree$ and $\eta$. The conforming part of $\Polyell{\degree}$ is the strict subspace
\begin{equation}
\label{DG-Poisson:conforming-part}
 \Polyell{\degree} \cap \SobH{\Domain}
 =
 \Lagr{\degree}
 =
 \{ s \in \Polyell{\degree} \mid \forall F \in \FacesM \; \Jump{s}{F}  \equiv 0  \}.
\end{equation}
Moreover, we easily see that
\begin{equation}
\label{Poisson:no-nondegeneracy}
 \emptyset
 \neq
 \Polyell{0}
 \subseteq
 \Polyell{\degree} \cap V^\perp,
\end{equation}
which precludes overconsistency in light of Lemma~\ref{L:a(.,E.)}.

% finding conditions for smoother
In order to obtain hints for a suitable choice of the smoothing operator, we invoke integration by parts and the piecewise structure of $\Polyell{p}$. Let $s,\sigma \in \Polyell{\degree}$ be arbitrary. On the one hand, piecewise integration by parts \eqref{pwIbP} yields
\begin{equation*}
 (s,\smt\sigma)_{1;\eta}
 =
 \sum_{K\in\Mesh} \int_K (-\Delta s) \smt\sigma
 +
 \sum_{F \in\FacesMint} \int_F \Jump{\nabla s}{F}\cdot \nF \smt\sigma
\end{equation*}
due to $E\sigma \in \SobH{\Domain}$. On the other hand, we want $\int_{\Omega} \GradM s \cdot \GradM \sigma=(s,\sigma)_1$ to appear in the discrete bilinear form. For this term, \eqref{jumps-avgs@bdy} and \eqref{pwIbP} give 
\begin{equation*}
\begin{aligned}
% \int_{\Omega} \GradM s \cdot \GradM \sigma
 (s,\sigma)_1
 =
 \sum_{K\in\Mesh} \int_K (-\Delta s) \sigma
 &+
 \sum_{F \in\FacesMint} \int_F \Jump{\nabla s}{F}\cdot \nF \Avg{\sigma}{F}
%\\
 &+
 \sum_{F \in\FacesM} \int_F
 \Avg{\nabla s}{F}\cdot \nF \Jump{\sigma}{F}.
\end{aligned}
\end{equation*}
A comparison of these two identities suggests that the smoothing operator $E$ should conserve certain moments on faces and elements and proves the following lemma.  Such moment conservation was already used in Badia et al.\ \cite[\S6]{Badia.Codina.Gudi.Guzman:14} to design a DG method for the Stokes problem with a partial quasi-optimality result for the velocity field as well as in \cite[\S\S3.2 and 3.3]{Veeser.Zanotti:17p2} to construct overconsistent Crouzeix-Raviart-like methods of arbitrary fixed order.
\begin{lemma}[Conservation of moments] 
\label{L:Poisson:moment-conservation}
Let $\degree\in\N$ and, for notational convenience, set $\Poly{-1}(K)=\emptyset$ for all $K\in\Mesh$. If the linear operator $\smt:\Polyell{\degree} \to \SobH{\Domain}$ satisfies
\begin{equation}
\label{DG-for-Poisson:conditions-for-E}
 \int_F q (\smt \sigma) 
 =
 \int_F q \Avg{\sigma}{F}
\quad\text{and}\quad
 \int_K r (\smt \sigma)  
 = 
 \int_K r \sigma 
\end{equation}
for all $F\in\FacesMint$, $q\in\Poly{\degree-1}(F)$, $K\in\Mesh$, $r\in \Poly{\degree-2}(K)$ and $\sigma\in\Polyell{\degree}$, then
\begin{equation*}
 (s, \smt\sigma)_{1;\eta} 
 =
 \int_\Domain \GradM s \cdot \GradM \sigma
 -
 \sum \limits_{F \in \FacesM}
	\int_F \Avg{\Grad s }{F} \cdot \nF \Jump{\sigma}{F} 
\end{equation*}
for all $s,\sigma \in S$.
\end{lemma}

% construction of smoother
We adapt the construction of the smoothing operators in \cite{Veeser.Zanotti:17p2} to the given setting and begin with a so-called bubble smoother, which employs the following weighted $L^2$-projections associated to faces and elements. For every interior face $F\in\FacesMint$, let $Q_F:\Leb{F} \to \Poly{\degree-1}(F)$ be given by
\begin{equation}
\label{Q_F}
\forall q \in \Poly{\degree-1}(F)
\quad
\int_{F} (Q_Fv) q \, \Phi_F = \int_{F} v q,
\end{equation}
where $\Phi_F := \prod_{z \in \LagNod{1}(F)} \Phi_z^1 \in \Lagr{d}$ is the face bubble function supported in the two elements containing $F$.  Moreover, for every mesh element $K \in \Mesh$, set $Q_K=0$ if $p=1$, otherwise let $Q_K: \Leb{K} \to \Poly{\degree-2}(K)$ be given by
\begin{equation}
\label{Q_K}
\forall r \in \Poly{\degree-2}(K)
\quad
\int_{K} (Q_Kv) r \, \Phi_K = \int_{K} v r,
\end{equation}
where $\Phi_K := \prod_{z\in\LagNod{1}(K)} \Phi^1_z \in \Lagr{d+1}$ is the element bubble function with support $K$. For $v \in H^1(\Mesh)$, we then define the global bubble operators
\begin{equation*}
\BubbOper[\Mesh,\degree] \sigma
:=
\sum_{K \in \Mesh} (Q_K v) \Phi_K,
\quad
\BubbOper[\FacesM,\degree] v
:=
\sum_{F \in \FacesMint} \sum_{z \in \LagNod{\degree-1}(F)}
 \big( Q_F \Avg{v}{F} \big)(z) \Phi_z^{\degree-1} \Phi_F,
\end{equation*}
where $\BubbOper[\FacesM,\degree]$ incorporates an extension by means of Lagrange basis functions in view of the partition of unity $\sum_{z \in \LagNod{\degree-1}(F)} \Phi_z^{\degree-1} = 1$. Their combination provides  the desired property and an extension of the operator with the same name in \cite{Veeser.Zanotti:17p2}.

\begin{lemma}[Bubble smoother]
\label{L:DG-bubble-smoother}
For $p\in\N$, the linear operator $\BubbOper[\degree]:\Polyell{\degree} \to \SobH{\Domain}$ defined by
\begin{equation*}
 \BubbOper[\degree] \sigma
 :=
 \BubbOper[\FacesM,\degree] \sigma
 +
 \BubbOper[\Mesh,\degree](\sigma - \BubbOper[\FacesM,\degree] \sigma)
\end{equation*}
satisfies \eqref{DG-for-Poisson:conditions-for-E} and the local stability estimate
\begin{equation*}
 \norm{\nabla \BubbOper[\degree] \sigma}_{\Leb{K}}
 \leq
 \frac{C_{d,\degree}}{\rho_K} \left(
  \sup_{r\in\Poly{\degree-2}(K)}
   \frac{\int_K \sigma r}{\norm{r}_{\Leb{K}}}
  + \sum_{F\in\Faces{K}}
	 \frac{|K|^{\frac{1}{2}}}{|F|^{\frac{1}{2}}}
	 \sup_{q\in\Poly{\degree-1}(F)}
	  \frac{\int_F \Avg{\sigma}{F} q}{\norm{q}_{\Leb{F}}}
 \right).
\end{equation*}
\end{lemma}

\begin{proof}
Proceed as in the proof of \cite[Lemma~3.8]{Veeser.Zanotti:17p2}.
\end{proof}

% stabilization
The factor $\rho_K^{-1}$ in the stability estimate in Lemma~\ref{L:DG-bubble-smoother} suggests that $\BubbOper[\degree]$ is not uniformly stable under refinement. The example in \cite[Remark~3.5]{Veeser.Zanotti:17p2} confirms this also for the current setting. However, since the bound involves lower order norms, we have the possibility to stabilize. This will be done with the help of the following variant $\AvgOper{\degree}:\Polyell{\degree} \to \Lagr{\degree}$ of nodal averaging. For every interior node $z\in\LagNodMint{\degree}$, fix some element $K_z \in \Mesh$ containing $z$ and set
\begin{equation}
\label{Averaging-def}
\textstyle
\AvgOper{\degree} \sigma
:=
\sum_{z\in\LagNodMint{\degree}} \sigma_{|K_z} (z) \Phi^\degree_z,
\qquad
\sigma \in \Polyell{\degree}.
\end{equation}
Clearly, $\AvgOper{\degree}\sigma(z) = \sigma (z)$ whenever $\sigma$ is continuous at $z\in\LagNodMint{\degree}$ and so $\AvgOper{\degree}$ is a projection onto $\Lagr{\degree}$. On the one hand, the operator $\AvgOper{q}$ is a restriction of Scott-Zhang interpolation \cite{Scott.Zhang:90} defined for broken $H^1$-functions and, on the other hand, it is a simplified variant of nodal averaging in that it requires only one evaluation per degree of freedom. Nodal averaging has been used in various nonconforming contexts, see, e.g., Brenner \cite{Brenner:96}, Karakashian/Pascal \cite{Karakashian.Pascal:03}, Oswald \cite{Oswald:93}. It maps into $\SobH{\Domain}$ along with the following error bound in terms of jumps.
\begin{lemma}[Simplified nodal averaging and $L^2$-norms of jumps]
\label{L:averaging-for-IP}
Let $p\in\N$, $\sigma \in \Polyell{\degree}$ piecewise polynomial, $K \in \Mesh$, and $z \in \LagNod{\degree}(K)$ be a Lagrange node. If $z \not\in \partial K$, then $\AvgOper{\degree}\sigma(z) = \sigma_{|K}(z)$, else
\begin{equation*}
%\label{Averaging-est}
	%
	%\begin{split}
 \snorm{\sigma_{|K}(z) - \AvgOper{\degree}\sigma(z)}
 \leq
	%	%
	%	\sum \limits_{F \in \FacesM, F \ni z}
	%	\snorm{\Jump{\sigma}{F} (z)}\\
	%	&\Cleq
	%
 C_{d,p} \sum \limits_{F \in \FacesM: F \ni z}
  \dfrac{1}{\snorm{F}^{\frac{1}{2}}} \norm{\Jump{\sigma}{F}}_{\Leb{F}}.
	%\end{split}
\end{equation*}
\end{lemma}

\begin{proof}
The `then'-part of the claim readily follows from the non-overlapping of elements in $\Mesh$. For the `else'-part, we first recall that $\Jump{\sigma}{F}_{|F}$ denotes the jump across the face $F$ and notice that its point values are well-defined. We thus can derive
\begin{equation*}
%\label{Averaging-est-1}
%
\snorm{\sigma_{|K}(z) - \AvgOper{\degree}(z)}
% =
%\snorm{ \dfrac{1}{\#\omega_z} 
%	\sum \limits_{K' \in \Mesh, K' \ni z} \sum \limits_{j=1}^{n'}
%	\Jump{\sigma}{F'_{j}}(z) }
%
 \leq
 \sum \limits_{F\in\FacesM:F \ni z} 
 \snorm{\Jump{\sigma}{F}_{|F}(z)}. 
\end{equation*}
with the help of the face-connectedness of stars in $\Mesh$; cf.\ \cite[(3.6)]{Veeser.Zanotti:17p2}. Therefore, the inverse estimate $\norm{\cdot}_{\Lebp{\infty}{F}} \leq C_{d,\degree} \snorm{F}^{-\frac{1}{2}} \norm{\cdot}_{\Leb{F}}$ in $\Poly{\degree}(F)$ finishes the proof.
\end{proof}

Stabilizing the bubble smoother $\BubbOper[\degree]$ with simplified nodal averaging $\AvgOper{\degree}$, we obtain a smoothing operator with the desired properties.

\begin{proposition}[Stable smoothing with moment conservation]
\label{P:IP-smoother}
The linear operator $E_{\degree}:\Polyell{\degree}\to\SobH{\Domain}$ given by
\begin{equation*}
 E_{\degree}\sigma
 :=
 \AvgOper{\degree}\sigma
 +
 \BubbOper[\degree](\sigma - \AvgOper{\degree}\sigma)
\end{equation*}
is invariant on $\Lagr{\degree}$, satisfies \eqref{DG-for-Poisson:conditions-for-E} and, for all $\sigma \in \Polyell{\degree}$,
\begin{equation*}
 \norm{\GradM (\sigma - E_p\sigma)}_{\Leb{\Domain}}
 \leq
 C_{d,\Shape_\Mesh,\degree}
  \norm{ h^{-\frac{1}{2}} \Jump{\sigma}{} }_{\Leb{\Sigma}} 
% \left(
%  \sum_{F \in \FacesM} h_F^{-1} \norm{\Jump{\sigma}{F}}_{\Leb{F}}^2
% \right)^{\frac{1}{2}}.
\end{equation*}
\end{proposition}

\begin{proof}
We adapt the proof of Propositions~3.9 to the current setting with jumps in the extended energy norm.

Clearly, the operator $E_\degree$ is well-defined and maps into $\SobH{\Domain}$. With $\AvgOper{\degree}$, also $E_p$ is a projection onto $\Lagr{\degree}$. We next show that $E_p$ conserves the moments in \eqref{DG-for-Poisson:conditions-for-E}. Given any $F\in\FacesMint$ and any $q \in \Poly{\degree-1}(F)$, we arrange terms to exploit that $\BubbOper[\degree]$ conserves moments and get
\begin{equation}
\label{Ep-and-face-moments}
 \int_F (E_\degree\sigma)q
 =
 \int_F (\BubbOper[\degree] \sigma)q
 +
 \underbrace{\int_F (\AvgOper{\degree}\sigma - \BubbOper[\degree] \AvgOper{\degree}\sigma)q}_{=0}
 =
\int_F \Avg{\sigma}{F}q.
\end{equation}
Arguing similarly, we obtain also that the element moments in \eqref{DG-for-Poisson:conditions-for-E} are conserved.

Finally, we turn to the claimed stability bound. Let $\sigma \in \Polyell{\degree}$ and write
\begin{equation*}
\label{Ep-bound1}
 \norm{\GradM (\sigma - E_p\sigma)}_{\Leb{\Domain}}
 \leq
 \norm{\GradM(\sigma - \AvgOper{\degree} \sigma)}_{\Leb{\Domain}}
 +
 \norm{\nabla \BubbOper[\degree](\sigma - \AvgOper{\degree} \sigma)}_{\Leb{\Domain}}.
\end{equation*}
In order to bound the right-hand side, we fix a mesh element $K \in \Mesh$ and consider the first term. Employing $\Phi_z^\degree{}_{|K}=\Psi_{K,z}^\degree$ and \eqref{Lagrange-basis:scaling} and then Lemma \ref{L:averaging-for-IP}, we obtain
\begin{equation}
\label{Ep-bound-2'}
\begin{split}
 \|\nabla (\sigma & - \AvgOper{\degree} \sigma)\|_{\Leb{K}} 
 \leq
 \sum \limits_{z \in \LagNod{\degree}(K)}
 \snorm{\sigma_{|K}(z) - \AvgOper{\degree}\sigma(z)} 
 \norm{\nabla\Phi_z^\degree}_{\Leb{K}}
\\
 &\leq 
 C_{d,\Shape_\Mesh,\degree} \sum \limits_{z \in \LagNod{\degree}(K)}
  \snorm{\sigma_{|K}(z) - \AvgOper{\degree}\sigma(z)} 
   \dfrac{\snorm{K}^{\frac{1}{2}}}{\rho_K}
\\
&\leq 
 C_{d,\Shape_\Mesh,\degree} \sum \limits_{z \in \LagNod{\degree}(K)} \;
  \sum \limits_{F' \in \FacesM, F' \ni z}
   \dfrac{\snorm{K}^{\frac{1}{2}}}{\rho_K\snorm{F'}^{\frac{1}{2}}}
    \norm{\Jump{\sigma}{F'}}_{\Leb{F'}}
\end{split}
\end{equation}
If $K'\in\Mesh$ contains a face $F'$ of the sum, then \eqref{adjacent-elements} implies
\begin{equation*}
  \dfrac{\snorm{K}^{\frac{1}{2}}}{\rho_K\snorm{F'}^{\frac{1}{2}}}
  \leq
  \frac{h_K}{\rho_K} \left(
   \frac{h_K^{d-2}}{\rho_{K'}^{d-1}}
  \right)^{\frac{1}{2}}
  \Cleq
  \rho_{K'}^{-\frac{1}{2}}
  \Cleq
  h_{F'}^{-\frac{1}{2}}.
\end{equation*}
Consequently, with the help of $\#\{K'\in\Mesh\mid K'\subseteq \omega_K\} \leq C_{d,\Shape_\Mesh}$, we arrive at
\begin{equation}
\label{Ep-bound2}
 \|\nabla (\sigma - \AvgOper{\degree} \sigma)\|_{\Leb{K}} 
 \Cleq
 \left(
  \sum_{F \in \FacesM, F\cap K\neq\emptyset}
    h_F^{-1} \norm{\Jump{\sigma}{F}}_{\Leb{F}}^2
 \right)^{\frac{1}{2}}.
\end{equation}
Next, consider the second term and observe that \eqref{intregation-bary-coord} gives
\begin{equation*}
 \sup_{r\in\Poly{\degree-2}(K)}
 \frac{\int_K (\sigma - \AvgOper{\degree}\sigma)r}{\norm{r}_{\Leb{K}}}
 \leq
 C_{d,\degree} \snorm{K}^{\frac{1}{2}}
 \sum_{z \in \LagNod{\degree}(\partial K)}
  \snorm{\sigma_{|K}(z) - \AvgOper{\degree}\sigma(z)}
\end{equation*}
and, for every $F \in \Faces{K}$,
\begin{equation*}
 \sup_{q\in\Poly{\degree-1}(F)}
 \frac{\int_F (\Avg{\sigma}{F} - \AvgOper{\degree}\sigma)q}{\norm{q}_{\Leb{F}}}
 \leq
 C_{d,\degree} \snorm{F}^{\frac{1}{2}}
 \sum_{K' \supset F}
 \sum_{z \in \LagNod{\degree}(F)}
  \snorm{\sigma_{|K'}(z) - \AvgOper{\degree}\sigma(z)}.
\end{equation*}
Inserting these two bounds in the stability estimate of Lemma~\ref{L:DG-bubble-smoother}, we find essentially the bound after the second inequality in \eqref{Ep-bound-2'} and so also
\begin{equation}
\label{Ep-bound3}
 \norm{\nabla \BubbOper[\degree] (\sigma - \AvgOper{\degree} \sigma)}_{\Leb{K}}
 \Cleq
 \left(
 \sum_{F \in \FacesM, F\cap K\neq\emptyset}
  h_F^{-1} \norm{\Jump{\sigma}{F}}_{\Leb{F}}^2
 \right)^{\frac{1}{2}}.
\end{equation}
We arrive at the claimed inequality by summing \eqref{Ep-bound2} and \eqref{Ep-bound3} over all $K\in\Mesh$, observing that the number of elements touching a given face is $\leq C_{d,\Shape_\Mesh}$.
\end{proof}

The smoothing operator $E_{\degree}$ in Proposition~\ref{P:IP-smoother} is computationally feasible. In fact, we have that 
\begin{itemize}
\item it suffices to know the evaluations $\langle f, \Phi^{\degree}_z\rangle$ for $z \in \LagNodMint{\degree}$ as well as $\langle f, \Phi^{\degree-1}_z\Phi_F\rangle$ for $F\in\FacesMint$, $z \in \LagNod{\degree-1}(F)$, and $\langle f, \Phi^{\degree-2}_z\Phi_K\rangle$ for $K\in\Mesh$, $z \in \LagNod{\degree-2}(K)$,
\item the support of each $E_{\degree}\Psi_{K,z}^\degree$ is contained in $\omega_z$,
\item the operators $Q_F$ and $Q_K$ in \eqref{Q_F} and \eqref{Q_K} can be implemented via matrices associated with a reference element and, for $d=2$, $Q_F$ can be diagonalized by means of Legendre polynomials.
\end{itemize}

% b_nip
After having found a suitable smoothing operator, we now choose the bilinear form $\bCons(\cdot,\cdot)$. Recall that, due to \eqref{Poisson:no-nondegeneracy}, the bilinear form $(\cdot,E_\degree\cdot)_{1;\eta}$ is degenerate and so $\bCons(\cdot,\cdot)$ needs to be nontrivial. There are several choices; see, e.g., Arnold et al.\ \cite{Arnold.Brezzi.Cockburn.Marini:02}. Here we shall discuss the interplay between $\smt_\degree$ and some of them.

\subsubsection*{A quasi-optimal NIP method}
One possibility to achieve nondegeneracy is to employ the jump penalization in $(\cdot,\cdot)_{1;\eta}$. If, in addition, we neutralize the downgrading of coercivity by $-\int_{\Sigma} \Avg{\Grad s}{F}\cdot \nF \Jump{q}{F}$ in $(\cdot,E_p\cdot)_{1;\eta}$, we reestablish the bilinear form of the nonsymmetric interior penalty (NIP) method introduced in \cite{Riviere.Wheeler.Girault:01}:
\begin{equation}
\label{DG-Poisson:bnip}
 b_\NIP
 :=
 (\cdot,E_p\cdot)_{1;\eta} + \bCons_\NIP
\quad\text{with}\quad
 \bCons_\NIP(s,\sigma)
 :=
% \sum_{F\in\FacesM} 
 \int_{\Sigma} %\left(
 \Jump{s}{F}\Avg{\Grad \sigma }{F}\cdot \nF
 +
% \sum_{F\in\FacesM} 
%  \int_{\Sigma}
 \frac{\eta}{h_F}\Jump{s}{F} \Jump{\sigma}{F}.
\end{equation}

\begin{lemma}[$b_\NIP$ and extended energy norm]
\label{L:NIP-energy-norm}
For any penalty parameter $\eta>0$, we have
\begin{equation*}
 \forall s,\sigma \in S
 \quad
 b_\NIP(s,s) \geq \snorm{s}_{1;\eta}^2
  \quad\text{and}\quad
 b_\NIP(s,\sigma)
 \leq
 \left(
  1 + \sqrt{\eta^{-1}\eta_*}%\frac{1}{2}}\eta_*^{\frac{1}{2}}
 \right) \snorm{s}_{1;\eta} \snorm{\sigma}_{1;\eta},
\end{equation*}
where $\eta_*>0$ depends on $d$, $\degree$, and $\Shape_\Mesh$.
\end{lemma}
Hence, if the penalty parameter $\eta$ is not too small, we may consider $\snorm{\cdot}_{1;\eta}$ with the same $\eta$ to be the discrete energy norm associated with $b_\NIP$. Remarkably, as $\eta\to\infty$, the coercivity and continuity constants tend to their respective counterparts of the limiting conforming Galerkin method in $\Lagr{\degree}$.

\begin{proof}
The coercivity bound holds by construction. For the continuity bound, we observe that, if $F\in\Faces{K}$ is a face of any $K \in \Mesh$, we have the inverse estimate $\norm{\cdot}_{\Leb{F}} \leq C_{d,\Shape_\Mesh,\degree} h_F^{-\frac{1}{2}} \norm{\cdot}_{\Leb{K}}$ in $\Poly{\degree-1}(K)$ and set $\eta_* :=(d+1)C_{d,\Shape_\Mesh,\degree}^2$. Then 
\begin{equation}
\label{discrete-normal-derivative-bound}
 \norm{ h^{\frac{1}{2}} \Avg{\Grad \sigma }{}}_{\Leb{\Sigma}}^2
 \leq
 \eta_* \norm{\GradM\sigma}_{\Leb{\Domain}}^2
\end{equation}
and the claimed continuity bound follows by standard steps.
\end{proof}

% new nip and its quasi-optimality
We thus arrive at $\app_\NIP = (\Polyell{\degree},b_\NIP,E_p)$, a \emph{new variant of the NIP method of order $p$} with the discrete problem
\begin{equation}
\label{qoSIP}
  U \in \Polyell{\degree}
\quad\text{such that}\quad
  \forall \sigma\in\Polyell{\degree} \;\,
  b_\NIP(s,\sigma) = \langle f, E_p\sigma \rangle.
\end{equation}
Since $b_\NIP = (\cdot,\cdot)_1$ and $\smt = \id$ on $\Lagr{\degree}$, this is a nonconforming Galerkin method. In contrast to the original NIP method, it applies to any load $f \in H^{-1}(\Domain)$ and has the following property.

\begin{theorem}[Quasi-optimality of $\app_\NIP$]
\label{T:qopt-NIP}
For any $\eta>0$, the method $\app_\NIP$ is $\snorm{\cdot}_{1;\eta}$-quasi-optimal for the Poisson problem \eqref{intro-Poisson} with constant $\leq \sqrt{1 + C_{d,\Shape_\Mesh,\degree} \eta^{-1}}$.
\end{theorem} 

\begin{proof}
After using the combination of Lemma \ref{L:Poisson:moment-conservation} and Proposition \ref{P:IP-smoother} in Theorem \ref{T:abs-qopt}, it remains to bound $\delta$. Let $v\in \SobH{\Domain}$, $\sigma \in \Polyell{\degree}$ and denote by $\Ritz_{\eta,\degree}$ the $(\cdot,\cdot)_{1;\eta}$-orthogonal projection onto $\Polyell{\degree}$. Then $\Jump{v}{}=0=\Jump{E_p\sigma}{}$ and the definition of $\Ritz_{\eta,\degree}$ imply
\begin{equation*}
 b_\NIP(\Ritz_{\eta,\degree} v,\sigma) - (v,E_p\sigma)_{1}
 =
 (\Ritz_{\eta,\degree} v - v,E_p\sigma-\sigma)_1
 +
 \int_{\Sigma}
  \Jump{\Ritz_{\eta,\degree}v - v}{F} \Avg{\Grad\sigma}{F}\cdot \nF,
\end{equation*}
whence Proposition~\ref{P:IP-smoother} and \eqref{discrete-normal-derivative-bound} yield
\begin{equation}
\label{DG-Poisson:bound-for-delta}
 |b_\NIP(\Ritz_{\eta,\degree} v,\sigma) - (v,E_p\sigma)_1|
 \leq
 C_{d,\Shape_\Mesh,\degree} \eta^{-\frac{1}{2}}
  \snorm{\sigma}_{1;\eta}
  \snorm{\Ritz_{\eta,\degree}v - v}_{1;\eta}.
\end{equation}
We thus conclude $\delta \Cleq \eta^{-\frac{1}{2}}$ with the help of the coercivity bound in Lemma~\ref{L:NIP-energy-norm}.
\end{proof}

\subsubsection*{A quasi-optimal SIP method}
%
% intro
The NIP bilinear form $b_\NIP$ arises in particular by enforcing coercivity. As an alternative, one can achieve symmetry by changing the sign of the first term in $\bCons_\NIP$. This leads to the SIP bilinear form $b_\SIP$; cf.\ \eqref{intro-sip-problem}. While $b_\SIP$ verifies the same continuity bound as $b_\NIP$, the coercivity bound can be replaced as follows. Inequality \eqref{discrete-normal-derivative-bound} implies
\begin{equation*}
  \left|
   \int_{\Sigma} \Jump{s}{} \Avg{\Grad\sigma}{}\cdot n
  \right|
  \leq
  \frac{1}{2} \sqrt{\eta_*\eta^{-1}} \left(
   \eta \norm{ h^{-\frac{1}{2}} \Jump{s}{} }_{\Leb{\Sigma}}^2
   +
   \norm{ \GradM \sigma }_{\Leb{\Omega}}^2
  \right),
\end{equation*}
from which we get
\begin{equation}
\label{SIP:coercivity}
 \forall s\in\Polyell{\degree}
\quad
 b_{\SIP}(s,s)
 \geq
 \alpha(\eta_*\eta^{-1}) \snorm{s}_{1;\eta}^2
\quad\text{with}\quad
 \alpha(t) = 1 - \sqrt{t}.
\end{equation}
Hence, if $\eta >\eta_*$, then the discrete problem
\begin{equation}
\label{SIP-Poisson:discrete-problem}
 U \in \Polyell{\degree}
 \quad\text{such that}\quad
 \forall \sigma \in \Polyell{\degree} \;\;
 b_\SIP(U,\sigma) = \langle f, \smt_\degree \sigma \rangle
\end{equation} 
is well-posed and gives rise to a \emph{new variant of the SIP method}, which is a nonconforming Galerkin method and denoted by $\app_\SIP$. The following theorem covers the results illustrated in the introduction \S\ref{S:introduction} and is proven as Theorem \ref{T:qopt-NIP}.

\begin{theorem}[Quasi-optimality of $\app_\SIP$]
\label{T:qopt-SIP}
For any $\eta > \eta_*$, the method $\app_\SIP$ is $\snorm{\cdot}_{1,\eta}$-quasi-optimal for \eqref{intro-Poisson} with constant $\leq \sqrt{1 + C_{d,\Shape_\Mesh,\degree} \big(\alpha(\eta_*/\eta)\eta\big)^{-1}}$.
\end{theorem} 

For $\eta\to\infty$, we again end up in C\'ea's lemma for the limiting conforming Galerkin method in $\Lagr{\degree}$.
\subsubsection*{High-order smoothing with first-order averaging}
Assume that $\degree\geq2$. The simplified averaging operator $\AvgOper{1}$ is defined also on $\Polyell{\degree}$ and so we may consider
\begin{equation}
\label{DG-Poisson:smoother-with-first-order-averaging}
 \widetilde{E}_\degree \sigma
 :=
 \AvgOper{1}\sigma + \BubbOper[\degree](\sigma -  \AvgOper{1}\sigma),
\quad
 \sigma \in \Polyell{\degree},
\end{equation}
which is cheaper to evaluate than $E_\degree$. In order to assess this idea, let us first check in which sense $A_1$ can provide stabilization.

\begin{lemma}[First-order averaging for higher order piecewise polynomials]
\label{DG-Poission-L:first-order-averaging}
Let $\degree\geq2$, $K \in \Mesh$, and $F \in \Faces{K}$. For all $z \in \LagNod{\degree}(K) \cap F$ and all $\sigma \in \Polyell{\degree}$, we have
\begin{equation*}
 \snorm{\sigma_{|K} (z) - \AvgOper{1}\sigma(z)}
 \leq
 C_{d,\degree} \left(
 \sum \limits_{F' \cap F \neq \emptyset }
  \dfrac{1}{\snorm{F'}}  \snorm{\int_{F'} \Jump{\sigma}{F'} }
 +
 \sum \limits_{K' \cap F \neq \emptyset}
  \dfrac{h_{K'}}{\snorm{K'}^{\frac{1}{2}}} \norm{\Grad \sigma}_{\Leb{K'}}
 \right),
\end{equation*}
where $F'$ and $K'$ vary, respectively, in $\FacesM$ and $\Mesh$. 
\end{lemma}

\begin{proof}
Given any $z \in \LagNod{\degree}(K) \cap F$, Lemma~3.1 in \cite{Veeser.Zanotti:17p2} ensures
\begin{equation}
\label{Ap:bound-with-mean-jumps}
 \snorm{\sigma_{|K} (z) - \AvgOper{\degree}\sigma(z)}
 \leq
 \sum \limits_{F' \ni z}
 \dfrac{1}{\snorm{F'}}  \snorm{\int_{F'} \Jump{\sigma}{F'} }
 +
 C_{d,\degree} \sum \limits_{K' \ni z}
 \dfrac{h_{K'}}{\snorm{K'}^{\frac{1}{2}}} \norm{\Grad \sigma}_{\Leb{K'}}
\end{equation}
We distinguish two cases, depending whether or not $z$ is a vertex.

\emph{Case 1:} $z\in\LagNod{1}(K)$. Then we have $\AvgOper{1} \sigma (z) = \AvgOper{\degree} \sigma (z)$ and the claimed estimate follows from \eqref{Ap:bound-with-mean-jumps}.

\emph{Case 2:} $z \in \LagNod{\degree}(K) \setminus \LagNod{1}(K)$. Since $\AvgOper{1} \sigma_{|F} \in \Poly{1}(F)$ and $\sum_{y\in\LagNod{1}(F)} \lambda_{y}^K = 1$, we may write
\begin{equation}
\label{Averaging-proof-1}
 \textstyle
 |\sigma_{|K} (z) - \AvgOper{1}\sigma(z)|
 \leq
 \sum_{y \in \LagNod{1}(F)} \left|
  \sigma_{|K}(z) - \AvgOper{1} \sigma (y)
  \right| \lambda_{y}^K(z)
\end{equation}
and, for any $y \in \LagNod{1}(F)$, 
\begin{equation*}
 \left| \sigma_{|K}(z) - \AvgOper{1} \sigma (y) \right|
 \leq
 \left| \sigma_{|K}(z) - \sigma_{|K}(y) \right|
 +
 \left| \sigma_{|K}(y) - \AvgOper{1} \sigma (y) \right|.  
\end{equation*}
As the second term of the right-hand side is already bounded in Case~1, it remains to bound the first term. Writing $c$ for the mean value of $\sigma$ in $K$, we deduce
\begin{align*}
 \snorm{ \sigma_{|K}(z) - \sigma_{|K}(y)}
 &\leq
 \snorm{ \sigma_{|K}(z) - c}
 +
 \snorm{ \sigma_{|K}(y) - c}
\\
 &\Cleq
 \snorm{F}^{-\frac{1}{2}} \norm{\sigma_{|K} -c}_{\Leb{F}}
 \Cleq
 h_K \snorm{K}^{-\frac{1}{2}} \norm{\Grad \sigma}_{\Leb{K}}
\end{align*}
with the help of an inverse estimate in $\Poly{\degree}(F)$ and \cite[Lemma~3]{Veeser:16}.
\end{proof}

Notice that, under the assumptions of Lemma~\ref{DG-Poission-L:first-order-averaging}, a bound solely in jump terms is not possible. Using Lemma~\ref{DG-Poission-L:first-order-averaging} in the proof of Proposition~\ref{P:IP-smoother}, we obtain the following properties of $\widetilde{E}_\degree$.

\begin{proposition}[Moment conservation with first-order averaging]
\label{P:IP-smoother-with-first-order-averaging}
The linear operator $\widetilde{\smt}_{\degree}$ from \eqref{DG-Poisson:smoother-with-first-order-averaging} is invariant on $\Lagr{1}$, satisfies \eqref{DG-for-Poisson:conditions-for-E} and, for all $\sigma \in \Polyell{\degree}$,
\begin{equation*}
 \norm{\GradM (\sigma - \widetilde{\smt}_\degree\sigma)}_{\Leb{\Domain}}
 \leq
 C_{d,\Shape_\Mesh,\degree} \left(
  \sum_{F \in \FacesM} h_F^{-d} \left|
	\int_F  \Jump{\sigma}{F} 
 \right|^2
 +
 \norm{ \GradM \sigma }_{\Leb{\Domain}}^2
 \right)^{\frac{1}{2}}.
\end{equation*}
\end{proposition}

Combining the new smoothing operator $\widetilde{\smt}_\degree$ with one of the previous bilinear forms $b_\VAR$, $\VAR\in\{\NIP,\SIP\}$, leads to a nonconforming method $\widetilde{\app}_\VAR$ with discrete problem
\begin{equation}
\label{VAR-Poisson:discrete-problem}
 U \in \Polyell{\degree}
 \quad\text{such that}\quad
 \forall \sigma \in \Polyell{\degree} \;\;
 b_\VAR(U,\sigma) = \langle f, \widetilde{\smt}_\degree \sigma \rangle, 
\end{equation} 
which is well-posed for all $\eta > \eta_\VAR$. Hereafter
\begin{equation*}
  \eta_\VAR
  :=
  \begin{cases}
   0, &\text{if }\VAR=\NIP,
  \\
   \eta_*, &\text{if }\VAR=\SIP,
  \end{cases}
\quad\text{and}\quad
  \alpha_\VAR(t)
  :=
  \begin{cases}
  1, &\text{if }\VAR=\NIP,
  \\
  1-\sqrt{t}, &\text{if }\VAR=\SIP.
  \end{cases}
\end{equation*}
As $\widetilde{\smt}_\degree$ is only invariant on the strict subset $\Lagr{1}$ of \eqref{DG-Poisson:conforming-part} for $\degree\geq2$, the method $\widetilde{\app}_\VAR$ is \emph{not} a nonconforming Galerkin method.
Nevertheless:

\begin{theorem}[Quasi-optimality of $\widetilde{\app}_\VAR$]
\label{T:qopt-VAR}
Let $\VAR\in\{\NIP,\SIP\}$. If $\eta > \eta_\VAR$, the method $\widetilde{\app}_\VAR$ is $\snorm{\cdot}_{1;\eta}$-quasi-optimal for the Poisson problem \eqref{intro-Poisson} with constant $\leq C_{d,\Shape_\Mesh,\degree} \sqrt{1+\big( \alpha_\VAR(\eta_*/\eta)\eta \big)^{-1}}$.
\end{theorem}

\begin{proof}
Proceed as in the proof of Theorem~\ref{T:qopt-NIP} or as indicated for Theorem~\ref{T:qopt-SIP}, replacing $\smt_\degree$ by $\widetilde{\smt}_\degree$. The only difference is that, in the derivation of the counterpart of \eqref{DG-Poisson:bound-for-delta}, we use
\begin{equation*}
 \sum_{F \in \FacesM} h_F^{-d} \left|
   \int_F  \Jump{\sigma}{F} 
 \right|^2
 \Cleq
 \sum_{F \in \FacesM} h_F^{-1}
 \int_F  \left| \Jump{\sigma}{F} \right|^2 
\end{equation*}
and obtain only
\begin{equation*}
 |b_\VAR(\Ritz_{\eta,\degree} v,\sigma) - (v,\widetilde{\smt}_\degree\sigma)_1|
 \leq
 C_{d,\Shape_\Mesh,\degree}
 \sqrt{1+\big( \alpha_\VAR(\eta_*/\eta)\eta \big)^{-1}}
 \snorm{\sigma}_{1;\eta}
 \snorm{\Ritz_{\eta,\degree}v - v}_{1;\eta}
\end{equation*}
because the stability bound in Proposition~\ref{P:IP-smoother-with-first-order-averaging} involves gradient terms.
\end{proof}
%
%%Mixed boundary conditions
%\begin{remark}[Mixed boundary conditions]
%	\label{R:Averaging-boundary}
%	%
%	A counterpart of Lemma \eqref{L:Averaging-H1} holds when homogeneous Dirichlet boundary conditions are imposed only on a closed subset $\Gamma \subseteq \partial \Domain$. In this case, we consider the space $S_{\Gamma}^\degree(\Mesh)$ of continuous piecewise polynomials of degree $\degree$ on $\Mesh$, with vanishing trace on $\Gamma$. The averaging operator $\AvgOper{\degree, \degree}$ can modified by imposing the validity of (REF) for all $z \in \LagNodM{\degree}$, $z \notin \Gamma$ and we have
%	\begin{equation*}
%	\label{Averaging-mixed-BCs}
%	%
%	\snorm{(\sigma - \AvgOper{\degree, \degree} \sigma)_{|K} (z)}
%	\Cleq
%	\sum \limits_{K' \ni z}
%	\dfrac{h_{K'}}{\snorm{K'}}
%	\norm{\Grad \sigma}_{\Lebp{1}{K'}}
%	+
%	\sum \limits_{F' \ni z, \: F' \nsubseteq \Gamma_N}
%	\dfrac{1}{\snorm{F'}}  \snorm{\int_{F'} \Jump{\sigma}{F'}} 
%	\end{equation*}
%	where $\Gamma_N := \partial \Domain \setminus \Gamma$. A similar result is possible for a modified version of $\AvgOper{\degree,1}$.  
%\end{remark}

%
%\subsubsection*{Quasi-optimal methods with weak interior penalty}
%

%
%
\subsection{A quasi-optimal and locking-free method for linear elasticity}
\label{S:CR-elasticity}
%
% goal
The goal of this section is to conceive a quasi-optimal and locking-free method for linear elasticity.

% linear elasticity
Given $\Domain \subseteq \R^d$ as in \S\ref{S:simplices-meshes}, we consider the displacement formulation of the linear elasticity problem with pure displacement boundary conditions: find $u \in \SobH{\Domain}^d$ such that
\begin{equation}
\label{LE:problem}
 - \Div \big( 2 \mu \SymGrad(u) + \lambda \Div(u) \big) = f
 \text{ in }\Omega,
\quad
 u =0 \text{ on }\partial\Omega.
\end{equation}
Hereafter $\SymGrad(v) := (\Grad v + \Grad v^T) /2$ is the symmetric gradient and $\mu$, $\lambda>0$ are the Lam\'{e} coefficients. We shall mostly suppress the dependencies on $\mu$ in the notation, while we trace the ones on $\lambda$.

Let $\Mesh$ be a mesh of $\Omega$ as in \S\ref{S:simplices-meshes} and, for $
\eta\geq0$, define
\begin{equation}
\label{a_lambda,eta}
\begin{gathered}
 a_{\lambda;\eta}(v,w)
 :=
 \int_\Domain \left( 2\mu \SymGradM (v) : \SymGradM(w)
 + \lambda \DivM v \DivM w \right)
 +
 \int_{\Sigma} \frac{\eta}{h} \Jump{v}{F} \Jump{w}{F},
\\
 \norm{v}_{\lambda;\eta}
 =
 a_{\lambda;\eta}(v,v)^{\frac{1}{2}}
\end{gathered}
\end{equation}
for $v,w \in H^1(\Mesh)^d$ and abbreviate $a_{\lambda;0}$ to $a_{\lambda}$. We aim at applying Theorem~\ref{T:abs-qopt} with the following setting:
\begin{equation}
\label{LE:setting}
 V = \SobH{\Domain}^d,
\quad
 S \subseteq \Polyell{1},
\quad
 \aext = a_{\lambda;\eta} \text{ on } \Vext = \SobH{\Domain}+S,
\end{equation}
where $S$ will be specified below, $\eta>0$, and the colon indicates the matrix scalar product $G:H = \sum_{j,\ell = 1}^{d} G_{j\ell} H_{j\ell}$. Notice that $a_{\lambda;\eta}$ is then a scalar product and \eqref{ex-prob} provides a weak formulation of \eqref{LE:problem}.

% moment conservation
We readily deduce the following counterpart of Lemma~\ref{L:Poisson:moment-conservation}.
\begin{lemma}[Moment conservation] 
\label{L:LE:moment-conservation}
If the linear operator $\smt:(\Polyell{1})^d \to \SobH{\Domain}^d$ satisfies
\begin{equation}
\label{LE:conditions-for-E}
 \forall \sigma\in(\Polyell{1})^d, F\in\FacesMint
\quad
 \int_F \smt\sigma 
 =
 \int_F \Avg{\sigma}{F},
\end{equation}
then, for all $s,\sigma \in (\Polyell{1})^d$,
\begin{align*}
 a_{\lambda;\eta}(s, \smt\sigma) 
 &=
 a_{\lambda}(s,\sigma)
 - \int_{\Sigma} \big(
    \Avg{ 2\mu \SymGradM(s) + \lambda \DivM(s) I}{F}
  \big) \, \nF \cdot \Jump{\sigma}{F}.
\end{align*} 

\end{lemma}

% choosing S
In \S\ref{S:DG-for-Poisson}, the impact on coercivity or symmetry of the counterpart of the term $\int_{\Sigma} (\Avg{ 2\mu \SymGradM(s) + \lambda \DivM(s) I}{F})\nF \cdot \Jump{\sigma}{F}$ was compensated with the help of $\bCons(\cdot,\cdot)$. Here we shall handle it with the choice of the discrete space $S$. More precisely, if we choose the Crouzeix-Raviart space
\begin{equation}
\label{LE:S}
 S = \CRS{}^d 
\quad\text{with}\quad
 \CRS{}
 :=
 \{ s \in \Polyell{1} \mid 
  \forall F \in \Faces \; \int_F \Jump{s}{F} = 0  \}
\end{equation}
with homogeneous boundary conditions, then this term vanishes because, on each face $F\in\FacesM$, the average $(\Avg{ 2\mu \SymGradM(s) + \lambda \DivM(s) I}{F}) \nF$ is a constant. Furthermore, $\int_F \sigma$, $F\in\FacesM$, is well-defined and equals the right-hand side of \eqref{LE:conditions-for-E}. The conforming part of $\CRS{}^d$ is
\begin{equation*}
 \CRS{}^d \cap \SobH{\Domain}^d = (\Lagr{1})^2,
\end{equation*}
which is a strict subspace for $\#\Mesh>1$. Finally, Arnold \cite{Arnold:93} shows that, for certain choices of $\Domain$ and $\Mesh$, there is a nonzero function
\begin{equation}
\label{Arnold-CR-function}
 s_0 \in \CRS{}^2 \setminus \{0\}
\quad\text{with}\quad
 \SymGradM(s_0) = 0 \text{ and }\DivM s_0 = 0,
\end{equation}
entailing, in contrast to the Poisson problem,
\begin{equation*}
%\label{LE:no-overconsistency}
 0 \neq s_0 \in \CRS{}^2 \cap (\SobH{\Domain}^2)^\perp
\end{equation*}
and overconsistency is in general ruled out by Lemma \ref{L:a(.,E.)}.

% smoother, discrete bilinear form and new method
As \eqref{LE:conditions-for-E} is the vector version of \eqref{DG-for-Poisson:conditions-for-E} for $\degree=1$, we can take the computionally feasible smoothing operator $\smt_1$ from Proposition \ref{P:IP-smoother} componentwise. We denote this vector version again by $\smt_1$. Since $a_{\lambda;\eta}(\cdot,\smt_1\cdot)$ may be degenerate in view of  \eqref{Arnold-CR-function}, we take
\begin{equation}
 b_{\HL}
 :=
 a_{\eta,\lambda} + d_{\HL}
\quad\text{with}\quad 
 \bCons_\HL(s,\sigma)
 =
 \int_{\Sigma} \frac{\eta}{h} \Jump{s}{F} \Jump{\sigma}{F}
\quad\text{with}\quad
 \eta>0,
\end{equation}
which is the discrete bilinear form in Hansbo and Larson \cite[eq.\ (26)]{Hansbo.Larson:03}. We thus introduce a \emph{new penalized Crouzeix-Raviart method} $\app_\HL=(\CRS{}^d,b_{\HL},E_1)$ given by the following discrete problem: find $U \in \CRS{}^d$ such that
\begin{equation}
\label{LE:M_LH}
 \int_{\Omega} \big(
  2 \mu \SymGradM(U) : \SymGradM(\sigma) + \lambda \DivM U \DivM \sigma
 \big)
 +
 \int_{\Sigma} \frac{\eta}{h} \Jump{s}{F} \Jump{\sigma}{F}
 =
 \langle f, E_1\sigma \rangle
\end{equation}
for all $\sigma \in \CRS{}^d$. The method $\app_\HL$ is a nonconforming Galerkin method. The modification of the right-hand side with respect to \cite{Hansbo.Larson:03} allows to apply $H^{-1}(\Domain)$-volume forces with the following property.

%Quasi-optimality
\begin{theorem}[Quasi-optimality of $\app_\HL$]
\label{T:qopt-elasticity}
The method $\app_\HL$ is $\norm{\cdot}_{\lambda;\eta}$-quasi-optimal for \eqref{LE:problem} with constant $\leq \sqrt{1 + C_{d,\Shape_\Mesh}(2\mu+\lambda)\eta^{-1}}$.
\end{theorem}
\begin{proof}
We first use Lemma~\ref{L:LE:moment-conservation} and Proposition~\ref{P:IP-smoother} for $\degree=1$ in Theorem~\ref{T:abs-qopt}. Then it remains to bound $\delta$ in item (iii) of Theorem~\ref{T:abs-qopt}. Let $v\in\SobH{\Domain}^d$, $\sigma \in \CRS{}^d$, and denote by $\Ritz_{\lambda;\eta}$ the $a_{\lambda;\eta}$-orthogonal projection onto $\CRS{}^d$. Lemma~\ref{L:LE:moment-conservation}, the definition of $\CRS{}$, $\Jump{v}{}=0=\Jump{\smt_1\sigma}{}$, and the definition of $\Ritz_{\lambda;\eta}$ imply
\begin{equation*}
 b_\HL(\Ritz_{\lambda;\eta}v,\sigma) - a_{\lambda}(v,\smt_1\sigma)
 =
 a_{\lambda}(\Ritz_{\lambda;\eta}v-v,\smt_1\sigma-\sigma)
\end{equation*}
and so Proposition~\ref{P:IP-smoother} yields
\begin{equation*}
 |b_\HL(\Ritz_{\lambda;\eta}v,\sigma) - a_{\lambda}(v,\smt_1\sigma)|
 \leq
 C_{d,\Shape_\Mesh} \sqrt{2\mu+\lambda} \, \eta^{-1/2}
  \norm{\Ritz_{\lambda;\eta}v-v}_{\lambda;\eta}
  \norm{\sigma}_{\lambda;\eta}.
\end{equation*}
Hence, we have $\delta \Cleq \sqrt{2\mu+\lambda} \, \eta^{-\frac{1}{2}}$ and the proof is finished.
\end{proof}

% sharpness of bound for quasi-optimality constant
The following remarks show that the upper bound of the quasi-optimality constant $\Cqopt$ in Theorem~\ref{T:qopt-elasticity} captures the correct asymptotic behavior not only for the conforming limit $\eta\to\infty$.

\begin{remark}[$\Cqopt$ as $\eta\to0$]
\label{R:elasticity-eta}
The degeneracy of the bilinear form $a_{\lambda;\eta}(\cdot,E_1\cdot)$ entails $\Cqopt\geq C_{\lambda} \eta^{-\frac{1}{2}}$. To see this, suppose that $s_0$ satisfies \eqref{Arnold-CR-function} and notice that identity \eqref{LE:conditions-for-E} and \cite[Lemma 3.2]{Veeser.Zanotti:17p2} guarantee that $\smt_1$ is injective. We then have that $\norm{E_1 s_0}_{\lambda;\eta} = C_\lambda \neq 0$ and $\norm{s_0}_{\lambda;\eta} = C\eta^{\frac{1}{2}}$. 	Hence, Theorem~\ref{T:abs-qopt} yields $\Cqopt \geq \Cstab \geq C_{\lambda} \eta^{-\frac{1}{2}}$.
\end{remark}

The following remark is closely connected with Linke  \cite[Section~2]{Linke:14} concerning incompressible flows.
\begin{remark}[Deterioration of $\Cqopt$ for nearly incompressible materials]
\label{R:elasticity-lambda}
The property
\begin{equation}
\label{Div-kernels}
 E_1\big( \{ s \in \CRS{}^d \mid \DivM s = 0 \} \big)
 \not\subseteq
 \{v \in \SobH{\Domain}^d \mid \Div v = 0 \} 
\end{equation}
results in $\Cqopt \geq C_\eta \lambda^{\frac{1}{2}}$. Indeed, if $s \in \CRS{}^d$ such that $\DivM s = 0$ and $\Div (E_1 s) \neq 0$, we have $\norm{s}_{\lambda;\eta} = C_\eta$ and $\norm{E_1 s}_{\lambda;\eta} \approx C \lambda^{\frac{1}{2}}$ as $\lambda\to\infty$ and so Theorem~\ref{T:abs-qopt} implies $\Cqopt \geq \Cstab \geq C_\eta \lambda^{\frac{1}{2}}$.

In order to verify \eqref{Div-kernels}, fix any face $F \in \FacesMint$ of a given mesh $\Mesh$. Let $\Psi_F$ the associated basis function in  $\CRS{}$ with $\int_{F'} \Psi_F = \delta_{FF'}$ for all $F' \in\FacesM$ and $\Psi_{F |K} = 0$ whenever $F \not\in \Faces{K}$. Then, appropriately picking the elements $K_z$ in the definition \eqref{Averaging-def} of $\AvgOper{1}$, we can arrange $\AvgOper{1} \Psi_F = 0$ and so $\smt_1 \Psi_F = \beta \Phi_F$ with some $\beta>0$ and $\Phi_F$ as in \eqref{Q_F}. Consider $\Psi_F t_F \in \CRS{}^d$, where $t_F$ is a unit tangent vector of $F$. On the one hand, we have $\DivM (t_F \Psi_F) = t_F \cdot \GradM \Psi_F = 0$ and, on the other hand, $\Div \smt_d (t_F \Psi_F) = \beta \Div(t_F \Phi_F) = \beta t_F \cdot \Grad \Phi_F \neq 0$. 
\end{remark}
 
% closeness to original HL method
It is instructive to shed additional light on the performance of $\app_\HL$ for nearly incompressible materials. First, recall that the space $\Lagr{1}$ shows locking whenever $\{s \in \Lagr{1} \mid \Div s = 0 \}$ provides poor approximation; see \cite{Brenner.Scott:08}. Hence the choice $\eta \approx \lambda$ will also result in poor approximation for large $\lambda$. For fixed penalty parameter $\eta>0$, the following lemma, which is also of interest by its own, will be useful. It quantifies the difference between the original method $\hat{\app}_\HL$ of Hansbo and Larson,
and its new variant $\app_\HL$. Recall that, if $f \in \Leb{\Domain}^d$, the discrete solution $\hat{U} \in \CRS{}^d$ of Hansbo and Larson is given by
\begin{equation}
\label{disc-prob-HansLars}
 \forall \sigma \in \CRS{}^d
\quad
 b_\HL(\hat{U}, \sigma)
 =
 \int_\Domain f \cdot \sigma.
\end{equation} 

\begin{lemma}[$\app_\HL$ and $\hat{\app}_\HL$]
\label{L:HLnew-HLoriginal}
Assume $f \in \Leb{\Domain}^d$ and let $U,\hat{U}\in\CRS{}^d$ verify \eqref{LE:M_LH} and \eqref{disc-prob-HansLars}, respectively. Then
\begin{equation*}
 \norm{U-\hat{U}}_{\lambda;\eta}
 \leq
 C_{d,\Shape_\Mesh} \min \left\{1,\eta^{-\frac{1}{2}} \right\} \left(
  \sum_{K \in \Mesh} h_K^2 \norm{f}_{\Leb{K}^d}^2
 \right)^{\frac{1}{2}}.
\end{equation*}
\end{lemma}

\begin{proof}
The definition of $U$ and $\hat{U}$ immediately give
\begin{equation*}
\label{est-ElastQopt-duality}
 \norm{U - \hat{U}}_{\lambda;\eta}
 =
 \sup \limits_{\norm{\sigma}_{\lambda;\eta}=1 } 
  \snorm{b_\HL(U-\hat{U}, \sigma)}
 =
 \sup \limits_{\norm{\sigma}_{\lambda;\eta}=1}
  \snorm{\int_\Domain f \cdot (\smt_1\sigma - \sigma)},
\end{equation*}
where $\sigma$ varies in $\CRS{}^d$. For any element $K\in\Mesh$, we have $\int_{\partial K} \smt_1\sigma - \sigma = 0$ implying the Poincar\'e inequality $\norm{\smt_1\sigma-\sigma}_{\Leb{K}^d} \Cleq h_K \norm{\Grad(\smt_1\sigma-\sigma)}_{\Leb{K}^d}$. Therefore,
\begin{align*}
 \snorm{\int_\Domain f \cdot (\smt_1\sigma - \sigma)}
% &\leq
% \sum_{K \in \Mesh} 
%  \norm{f}_{\Leb{K}^d} \norm{\smt_1\sigma-\sigma}_{\Leb{K}^d}
%\\
% &
 \Cleq
  \left(
 \sum_{K \in \Mesh} h_K^2 \norm{f}_{\Leb{K}^d}^2
 \right)^{\frac{1}{2}}
 \norm{\GradM(\smt_1\sigma-\sigma)}_{\Leb{\Domain}^d}.
\end{align*}
Hence, Proposition \ref{P:IP-smoother} and \cite[Proposition~3.3]{Veeser.Zanotti:17p2} followed by Brenner \cite[Theorem~3.1]{Brenner:04} finish the proof.
\end{proof}
We readily see from this proof that the asymptotic closeness of $U$ and $\hat{U}$ could be increased by requiring that the smoothing operator conserves also element moments.

% locking-free
A consequence of Lemma~\ref{L:HLnew-HLoriginal} is the following equivalence concerning the asymptotic error bounds
\begin{equation*}
 \norm{u-U}_{\lambda;\eta}
 \leq
 C h \norm{f}_{\Leb{\Domain}^d},
\quad
  \norm{u-\hat{U}}_{\lambda;\eta}
 \leq
 \hat{C} h \norm{f}_{\Leb{\Domain}^d}
\end{equation*}
with best constants $C$ and $\hat{C}$ for all $h := \max_{K \in \Mesh} h_K$ and $f \in \Leb{\Domain}^d$:
\begin{equation}
\label{locking-iff}
 \text{$C$ is independent of $\lambda$}
 \iff
 \text{$\hat{C}$ is independent of $\lambda$}.
\end{equation}
Therefore, the robustness result \cite[Theorem~3.1]{Hansbo.Larson:03}, which ensures that $\hat{C}$ is independent of $\lambda$ for polygons $\Omega\subseteq\R^2$, %in a similar fashion as Brenner and Sung \cite{Brenner.Sung:92}
carries over to $\app_\HL$. In summary, for smooth volume forces, the method $\app_\HL$ is locking-free. The non-robustness of the quasi-optimality constant is thus due to rough volume forces, including forces for which the locking-free nonconforming methods in Falk \cite{Falk:91}, Brenner and Sung \cite{Brenner.Sung:92}, and Hansbo and Larson \cite{Hansbo.Larson:03} are not defined.

Let us conclude this section with a remark on the generalization to order $\degree \geq 2$, where $\CR$ is replaced by its higher order counterpart from $\CRS{\degree}$ from Baran and Stoyan \cite{Baran.Stoyan:06}. This case is of different nature. In fact, the Korn inequalities of Brenner \cite{Brenner:04}  ensure that $\norm{\cdot}_{\lambda;\eta}$ is a norm on $\SobH{\Domain}^d+\CRS{\degree}^d$ even for $\eta=0$. This allows to construct overconsistent methods with the help of $\smt_\degree$ from Proposition~\ref{P:IP-smoother}.

\subsection{A quasi-optimal $C^0$ interior penalty method for the biharmonic problem}
\label{S:IP-fourth-order}
%
% intro

%biharmonic problem
In this subsection, we introduce a new $C^0$ interior penalty method for the biharmonic problem with clamped boundary conditions,
\begin{equation}
\label{biharmonic-problem}
 \Lapl^2 u = f \text{ in }\Omega,
\quad
 u = \partial_n u = 0 \text{ on }\partial\Omega,
\end{equation}
and prove its quasi-optimality. We let $\Omega$ and $\Mesh$ be as in \S\ref{S:simplices-meshes} with $d=2$. %
Jumps and averages of vector- and matrix-valued maps are intended componentwise. Consequently, if $v \in H^2(\Mesh)$, then $\Jump{\Grad v}{F} \cdot \nF$ and $\Avg{\Grad v}{F} \cdot \nF$ indicate, respectively, the jump and the average of the normal derivative of $v$ on the skeleton $\Sigma$.  
We write also $\Jump{\partial^2 v/ \partial \nF^2}{F}$ and $\Avg{\partial^2 v/ \partial \nF^2}{F}$ in place of $(\Jump{D^2 v}{F} \nF) \cdot \nF$ and $(\Avg{D^2 v}{F} \nF) \cdot \nF$, respectively. Given $\eta\geq0$, set
\begin{equation}
\label{extended-H2-seminorm}
\begin{gathered}
 (v, w)_{2;\eta}
 :=
 \int_{\Domain} \HessM v: \HessM w
 +
 \int_{\Sigma} %^{max}\sum \limits_{F \in \FacesM}
  \dfrac{\eta}{h_F}
  %\int _F
  \left ( \Jump{\Grad v }{F} \cdot \nF \right )
  \left (\Jump{\Grad w }{F}\cdot \nF \right),
\\
 \snorm{v}_{2;\eta}
 :=
 (v,v)_{2;\eta}^{\frac{1}{2}}.
\end{gathered}
\end{equation}
for $v,w \in H^2(\Mesh)$ and abbreviate $(\cdot,\cdot)_{2;0}$ to $(\cdot,\cdot)_2$. Recalling \eqref{Skp}, consider the following setting for Theorem~\ref{T:abs-qopt}:
\begin{equation}
\label{bihamonic:setting}
\begin{gathered}
 V = \SobbH{\Domain},
\quad
 S = \Lagr{2},
\quad
 \aext = (\cdot,\cdot)_{2;\eta}
 \text{ on } \Vext = \SobbH{\Domain} + \Lagr{2}.
\end{gathered}
\end{equation}
For $\eta>0$, the bilinear form $(\cdot,\cdot)_{2;\eta}$ is a scalar product on
\begin{equation}
\label{biharmonic-CO:gradient-jump}
 \SobbH{\Domain}+\Lagr{2}
 \subseteq
 \{ v\in H^2(\Mesh) \mid \Jump{\Grad v}{}\cdot n^\perp = 0 \}
\end{equation}
and the abstract problem \eqref{ex-prob} with \eqref{bihamonic:setting} is a weak formulation of the biharmonic problem \eqref{biharmonic-problem}. The conforming part of $\Lagr{2}$ is the strict subspace
\begin{equation}
\label{biharmonic-C0:conforming-part}
 \Lagr{2} \cap \SobbH{\Domain}
 =
 \{ s \in \Lagr{2} \mid \Jump{\Grad s}{} \cdot n = 0 \},
\end{equation}
which may be even trivial; cf.\ \cite[Remark~3.11]{Veeser.Zanotti:17p2}.  Finally, we have
\begin{equation}
\label{biharmonic-C0:nontrivial-complement}
 \{0\} \neq \Lagr{1} \subseteq \Lagr{2}\cap \SobbH{\Domain}^\perp
\end{equation}
and, therefore, Lemma \ref{L:a(.,E.)} rules overconsistency out.

%The method by Brenner and Sung
Let us turn to the choice of the smoothing operator. Interestingly, Brenner and Sung \cite{Brenner.Sung:05} propose a $C^0$ interior penalty method $\BS{\app}$ involving a smoothing operator based upon averaging. In contrast to similar methods,  $\BS{\app}$  is well-defined for general loads $\ell \in H^{-2}(\Domain)$, fully stable according to Theorem~\ref{T:abs-qopt}~(i), and, for any $\alpha>0$ and all $\ell \in H^{-2+\alpha}(\Domain)$, its error in $\snorm{\cdot}_{2;\eta}$ with a suitable $\eta$ decays at the optimal rate $\alpha$. Nevertheless, $\BS{\app}$ is not guaranteed to be quasi-optimal with respect to $\snorm{\cdot}_{2;\eta}$, because it is not designed to be fully algebraically consistent.

To devise a method ensuring full algebraic consistency, we proceed as before and derive the following counterpart of Lemma~\ref{DG-for-Poisson:conditions-for-E} with the help of \eqref{pwIbP}.

\begin{lemma}[Moment conservation]
\label{L:biharmonic-moment-conversation}
If the smoothing operator $\smt:\Lagr{2} \to \SobbH{\Domain}$ satisfies
\begin{equation}
\label{biharmonic:moments}
 \forall \sigma \in \Lagr{2}, F \in \FacesMint
\quad
 \int_F \Grad \smt \sigma
 =
 \int_F \Avg{\Grad \sigma}{F},
\end{equation}
then
\begin{equation*}
%\label{bE-biharmonic}
%
 \forall s,\sigma \in \Lagr{2}
\quad
 (s, \smt\sigma)_{2;\eta} 
 =
 \int_\Domain \HessM s : \HessM \sigma
 -
 %\sum \limits_{F \in \FacesM}
 \int_\Sigma \Avg{ \dfrac{\partial^2 s}{\partial \nF^2} }{F}
  \Jump{\Grad \sigma }{F}\cdot \nF. 
\end{equation*}
\end{lemma}
Thanks to $\Lagr{2}+\SobbH{\Domain} \subseteq C^0(\overline{\Domain})$ and the fundamental theorem of calculus, we may ensure the conservation  \eqref{biharmonic:moments} of the mean gradients on faces by
\begin{equation}
\label{biharmonic:moments'}
 \forall z \in \LagNodMint{1}\;\;
 \smt\sigma(z) = \sigma(z)
\quad\text{and}\quad
  \forall F \in \FacesMint \;\;
 \int_F \Grad \smt \sigma \cdot \nF
 =
 \int_F \Avg{\Grad \sigma}{F} \cdot \nF.
\end{equation}
The smoothing operator for Morley functions in \cite{Veeser.Zanotti:17p2} verifies these new requirements. We adapt its construction to the current setting, focusing on the modifications only. Let us begin with the (simplified) averaging operator mapping into the Hsieh-Clough-Tocher (HCT) space
\begin{equation*}
 \HCT 
 :=
 \{ s \in C^1(\overline{\Domain})  \mid 
  \forall K \in \Mesh \;
  s_{|K} \in C^1(K) \cap \Poly{3}(\Mesh_K),
  s = \partial_n s = 0 \text{ on }\partial\Omega
 \},
\end{equation*}
where $\Mesh_K$ stands for the triangulation obtained by connecting each vertex of the triangle $K$ with its barycenter $m_K$. For each vertex $z\in\LagNodMint{1}$ and edge $F\in\FacesMint$, we pick elements $K_z, K_F\in\Mesh$ containing $z$ or $F$, respectively, and define
\begin{equation}
\label{HCT-smoother}
 A_{\hct}\sigma
 :=
 \sum_{z \in \LagNodMint{1}} \left(
  \sigma(z) \Upsilon_z^0
  + \sum_{j=1}^2 \partial_{j} \big( \sigma_{|K_z} \big)(z) \Upsilon_z^j
 \right)
 + \sum_{F \in \FacesMint}
  \frac{\partial\big( \sigma_{|K_{F}} \big)}{\partial \nF}(m_F)
  \Upsilon_{F},
\end{equation}
where $\Upsilon_z^j$ with $z \in \LagNodMint{1}$, $j \in \{ 0, 1, 2\}$ and $\Upsilon_{F}$ with $F \in \FacesMint$ form the nodal basis of $\HCT$. Next, we introduce the bubble smoother. Given any interior edge $F \in \FacesMint$, let $K_1, K_2 \in \Mesh$ be the two elements such that $F = K_1 \cap K_2$. Considering their barycentric coordinates $(\lambda_z^{K_i})_{z \in \LagNod{1}(K_i)}$, $i=1,2$, as first-order polynomials on $\R^2$, set
\begin{equation*}
\bar{\phi}_F
:=
\frac{30}{\snorm{F}} \phi_F
\quad\text{with}\quad
\phi_F
:=
\begin{cases}
\prod \limits_{z \in \LagNod{1}(F)}
\left( \lambda_z^{K_1} \lambda_z^{K_2}\right)^2
&\text{in } K_1 \cup K_2,
\\
0 &\text{in } \Domain \setminus(K_1 \cup K_2),
\end{cases}
\end{equation*}
and define $\zeta_F(x) := (x-m_F) \cdot \nF$ for  $x \in \R^2$. Then 
$
%\label{biharmonic:face-bubble}
 \bar{\Phi}_{n_F}
 :=
 \zeta_F \bar{\phi}_F
$ %\end{equation}
is in $\SobbH{\Domain}$ and satisfies
$%\begin{equation*}
% \forall F' \in \FacesMint
%\quad
\int_{F'} \nabla \bar{\Phi}_{n_F} \cdot n_{F'}
=
\int_{F'} n_F \cdot n_{F'} \bar{\phi}_{F}
=
\delta_{F,F'}
$ %\end{equation*}
for all $F' \in \FacesMint$ thanks to \eqref{intregation-bary-coord}. Hence, 
\begin{equation*}
\label{MR-bubble-smoother}
\BubbOper[\partial_n] \sigma
:=
\sum_{F \in \FacesMint} 
\left( \int_F \Avg{\Grad\sigma}{F} \cdot \nF \right) \bar{\Phi}_{n_F}
\end{equation*} 
maps $\Lagr{2}+\HCT$ into $\SobbH{\Domain}$, verifying $\BubbOper[\partial_n]\sigma(z) = 0$ for all $z\in\LagNodMint{1}$ and the second part of \eqref{biharmonic:moments'}. The combination of bubble smoother and averaging thus yields the desired moment conservation in a stable manner.

\begin{proposition}[Stable smoothing with moment conservation]
\label{biharmonic-C0-P:stable-smoothing}
The linear operator  $\smt_{\mathrm{C0}}:\Lagr{2}\to\SobbH{\Domain}$ given by
\begin{equation*}
 \smt_{\mathrm{C0}}\sigma
 :=
 \AvgOper{\hct}\sigma
 + \BubbOper[\partial_n](\sigma - \AvgOper{\hct}\sigma)
\end{equation*}
is invariant on $\Lagr{2}\cap\SobbH{\Domain}$, verifies \eqref{biharmonic:moments} and, for all $\sigma\in\Lagr{2}$,
\begin{equation*}
 \norm{\HessM (\sigma - \smt_{\mathrm{C0}}\sigma)}_{\Leb{\Domain}}
 \leq
 C_{\Shape_\Mesh} \left(
  \sum_{F \in \FacesM} h_F^{-1} \norm{\Jump{\Grad\sigma}{F}\cdot n}_{\Leb{F}}^2
\right)^{\frac{1}{2}}.
\end{equation*}
\end{proposition}

\begin{proof}
We proceed as in \cite[Proposition~3.16]{Veeser.Zanotti:17p2} with the following difference. For the HCT averaging, we use the bound 
\begin{equation*}
 \snorm{\Grad(\sigma_{|K})(z) - \Grad \AvgOper{\HCT}\sigma(z)}
 \leq
 C \sum \limits_{F \in \FacesM: F \ni z}
  h_F^{-\frac{1}{2}} \norm{\Jump{\Grad\sigma}{F} \cdot n}_{\Leb{F}},
\end{equation*}
which follows along the lines of the proof of Lemma~\ref{L:averaging-for-IP} and from \eqref{biharmonic-CO:gradient-jump}.
\end{proof}

%Discrete bilinear form
It remains to choose the bilinear form $\bCons(\cdot,\cdot)$. In view of \eqref{biharmonic-C0:nontrivial-complement}, we need to establish nondegeneracy, for example in the vein of the extended energy norm $\snorm{\cdot}_{2;\eta}$. Requiring also symmetry for the resulting discrete bilinear form then leads to 
\begin{equation*}
%\label{bCons-biharmonic}
%
 \BS{\bCons}(s, \sigma)
 =
 \int_{\Sigma}
 \Jump{\Grad s }{F}\cdot \nF 
 \left( 
  -
  \Avg{\dfrac{\partial^2 \sigma}{\partial \nF^2}}{F}
  +
  \dfrac{\eta}{h} 
  \Jump{\Grad \sigma }{F}\cdot \nF 
 \right) 
\end{equation*}
and the discrete bilinear form of Brenner and Sung \cite{Brenner.Sung:05}:
\begin{equation}
\label{biharmonic-C0:discrete-bilinear-form}
\begin{aligned}
 \BS{b}(s,\sigma)
 =
 (s,\sigma)_{2;\eta}
 -
 \int_\Sigma \left( 
\Avg{\dfrac{\partial^2 s}{\partial \nF^2}}{F}
\Jump{\Grad \sigma }{F}\cdot \nF
+
\Jump{\Grad s }{F}\cdot \nF
\Avg{\dfrac{\partial^2 \sigma}{\partial \nF^2}}{F}
\right). 
\end{aligned}
\end{equation}
Similarly to the SIP bilinear form, there is $\eta_* > 0$ depending on $\Shape_\Mesh$ such that
\begin{equation}
\label{biharmonic-C0:inverse-estimate}
 \norm{ h^{-\frac{1}{2}}
  \Avg{\partial^2\sigma/\partial^2 n}{}}_{\Leb{\Sigma}}
 \leq
 \eta_*
 \norm{\HessM\sigma}_{\Leb{\Domain}}
\end{equation}
and therefore $\BS{b}$ is $\snorm{\cdot}_{2,\eta}$-coercive with constant $\sqrt{\alpha(\eta_*/\eta)}$ whenever $\eta>\eta_*$; cf.\ \eqref{SIP:coercivity} and \cite[Lemma 7]{Brenner.Sung:05}. Under this assumption, the discrete problem
\begin{equation}
\label{biharmonic-C0:discrete-problem}
 U \in \Lagr{2} \text{ such that }
 \forall\sigma\in\Lagr{2} \;\;
 \BS{b}(U,\sigma) = \langle f, \smt_{\mathrm{C0}}\sigma \rangle
\end{equation}
is well-posed and introduces a \emph{new $C^0$ interior penalty method} $\app_{\mathrm{C0}}$  for the biharmonic problem \eqref{biharmonic-problem}.
Inspecting $\BS{b}$, $\smt_{\mathrm{C0}}$ and recalling Proposition~\ref{biharmonic-C0-P:stable-smoothing}, we see that 
$\app_{\mathrm{C0}}=(\Lagr{2},\BS{b},\smt_{\mathrm{C0}})$ is a nonconforming Galerkin method with a computationally feasible smoothing operator. It differs from the original method of Brenner and Sung \cite{Brenner.Sung:05} in the choice of the smoother and the following property.

\begin{theorem}[Quasi-optimality of $\app_{\mathrm{C0}}$]
\label{biharmonic-C0-T:quasi-optimality}
For any penalty parameter $\eta>\eta_*$, the method $\app_{\mathrm{C0}}$ is $\snorm{\cdot}_{2;\eta}$-quasi-optimal for the biharmonic problem \eqref{biharmonic-problem} with constant
$\leq \sqrt{1+C_{\Shape_\Mesh}\big( \alpha(\eta_*/\eta)\eta \big)^{-1}}$.
\end{theorem}

\begin{proof}
Assume $\eta>\eta_*$. Hence $\BS{b}$ is coercive and Theorem \ref{T:abs-qopt} applies. After making use of Lemma \ref{L:biharmonic-moment-conversation}, Proposition \ref{biharmonic-C0-P:stable-smoothing} and \eqref{biharmonic-C0:conforming-part}, it remains to bound $\delta$ in (iii) of Theorem \ref{T:abs-qopt}. To this end, we let $\Ritz_{\eta}$ denote the $(\cdot,\cdot)_{2;\eta}$-orthogonal projection onto $\Lagr{2}$ and derive, for all $v\in\SobbH{\Domain}$ and $\sigma\in\Lagr{2}$,
\begin{equation*}
 \BS{b}(\Ritz_{\eta}v,\sigma) - (v,\smt_{\mathrm{C0}}\sigma)_{2}
 =
 (\Ritz_{\eta}v - v, \smt_{\mathrm{C0}}\sigma - \sigma)_2
 -
 \int_{\Sigma}
  \Jump{\Ritz_{\eta} v -v}{F} \cdot n \Avg{\frac{\partial^2\sigma}{\partial n}}{}
%\end{aligned}
\end{equation*}
with the help of $\Jump{\Grad\smt_{\mathrm{C0}}\sigma}{} = 0 = \Jump{\Grad v}{}$, Lemma~\ref{biharmonic:moments} and the definition of $\Ritz_{\eta}$.
Consequently, Proposition~\ref{biharmonic-C0-P:stable-smoothing} and \eqref{biharmonic-C0:inverse-estimate} yield
\begin{equation*}
 |\BS{b}(\Ritz_{\eta}v,\sigma) - (v,\smt_{\mathrm{C0}}\sigma)_{2}|
 \leq
 C_{\Shape_\Mesh} \eta^{-\frac{1}{2}} \snorm{\Ritz_{\eta}v - v}_{2;\eta}
  \snorm{\sigma}_{2;\eta}.
\end{equation*}
The coercivity of $\BS{b}$ thus implies $\delta^2 \Cleq \big( \alpha(\eta_*/\eta)\eta \big)^{-1}$ and the proof is finished.
\end{proof}

The presented approach can be extended to design quasi-optimal methods of order $\degree \geq 3$. Perhaps the simplest manner is to keep the HCT averaging $\AvgOper{\hct}$ and to construct a higher order version of the bubble smoother similar to $\BubbOper[\degree]$ in \S\ref{S:DG-for-Poisson}. This does not result in a nonconforming Galerkin method, but achieves quasi-optimality.  

%---BIBLIOGRAFIA-------------------%

%\nocite{*}
%\bibliographystyle{siam}
%\bibliography{Bibliografia}
%\bibliography{/home/veeser/Unimibox/myTeX/av}

\end{document}